\documentclass[preprint,12pt]{article}
\usepackage[top=1.0in, bottom=1.0in, left=1in, right=1in]{geometry}
\usepackage{amsmath,amsfonts,amssymb,amsthm}
\usepackage{mathrsfs,dsfont}
\usepackage{graphicx}
\usepackage{colortbl,dcolumn}
\usepackage{psfrag}
\usepackage{booktabs}
\usepackage{cite}
\usepackage{url}
\usepackage{hyperref}
\usepackage{subfigure}
\allowdisplaybreaks[4]
\numberwithin{equation}{section}

\newcommand{\R}{\mathbb{R}}

\newcommand{\E}{\mathbb{E}}

\newtheorem{theorem}{Theorem}[section]
\newtheorem{definition}[theorem]{Definition}
\newtheorem{lemma}[theorem]{Lemma}
\newtheorem{proposition}[theorem]{Proposition}
\newtheorem{corollary}[theorem]{Corollary}

\newtheorem{assumption}[theorem]{Assumption}

\begin{document}
\title{Strong order one-half convergence of a coupled tamed Euler--Peano scheme
for reflected stochastic differential equations with super-linearly growing coefficients\footnotemark[1]}

\footnotetext{\footnotemark[1] This work was supported by National Natural Science Foundation of China (Nos. 12501581, 12201552), Yunnan Fundamental Research Projects (No. 202601AT070161), and the Disciplinary Funding of Central University of Finance and Economics.}

\author{Ziheng Chen\footnotemark[2],
	    \quad Caiyun Hua\footnotemark[3],
		\quad Meng Cai\footnotemark[4]}
\footnotetext{\footnotemark[2] School of Mathematics and Statistics, 
Yunnan University, Kunming, Yunnan, 650500, China. Email: czh@ynu.edu.cn.}

\footnotetext{\footnotemark[3] School of Mathematics and Statistics, Yunnan University,
Kunming, Yunnan, 650500, China. Email: 12024113042@stu.ynu.edu.cn.}

\footnotetext{\footnotemark[4] School of  Statistics and Mathematics, 
Central University of Finance and Economics, Beijing, 100081, China. 
Email: mcai@lsec.cc.ac.cn. Corresponding author.}

%\date{}

\maketitle

\begin{abstract}
      {\rm\small  We study strong numerical approximations for reflected stochastic differential equations in possibly unbounded convex domains with super-linearly growing drift and diffusion coefficients. Under a coupled monotonicity condition and polynomial local Lipschitz assumptions, we first establish the well-posedness of the reflected SDE and derive uniform moment bounds for its solution. We then introduce a coupled tamed Euler--Peano scheme, in which the drift and the squared diffusion coefficient are tamed by a common factor and the resulting Euler--Peano path is corrected through the Skorokhod problem. This common taming factor preserves the drift--diffusion coercivity structure and yields uniform moment estimates for the numerical solution. We prove strong convergence of order $1/2$ for both the constrained state process and the boundary regulator, thereby recovering the standard Euler-type strong order in this reflected setting. Numerical experiments for a reflected stochastic Ginzburg--Landau type system illustrate the constraint preservation of the scheme and support the theoretical convergence rate.}

      \vspace{1em}\textbf{AMS subject classification: }{\rm\small 60H10, 60H35, 65C30}

      \vspace{1em}\textbf{Key Words:}{\rm\small} reflected stochastic differential equation; coupled tamed Euler--Peano method; super-linearly growing coefficients; strong convergence of order $1/2$; boundary regulator; Skorokhod problem
\end{abstract}

\section{Introduction}\label{sec:introduction}
We study strong numerical approximations for the following reflected stochastic differential equation (RSDE) with non-globally Lipschitz coefficients in a possibly unbounded nonempty open convex domain $D \subset \mathbb{R}^{d}$:
\begin{equation}\label{eq:reflected-sde-intro}
      X(t)
      =
      x_{0} + \int_{0}^{t} b\big(X(s)\big)\,ds
      +
      \int_{0}^{t} \sigma\big(X(s)\big)\,dW(s)
      +
      K(t), \quad t \in [0,T].
\end{equation}
Here $x_{0}\in D$, $W$ is an $m$-dimensional Brownian motion, and $K$ denotes the boundary regulator which keeps the state process in $\overline{D}$ through the Skorokhod reflection mechanism. This boundary-supported regulation mechanism makes RSDEs a natural framework for modelling constrained stochastic dynamics, with typical applications in queueing networks, constrained diffusion processes, stochastic control, mathematical finance, interacting particle systems, and stochastic variational inequalities; see, e.g., \cite{tanaka1979stochastic, lions1984stochastic, saisho1987stochastic, harrison1981reflected, dupuis1991lipschitz, glaister1988shock, ramanan2006reflected, pilipenko2014introduction}. From the numerical point of view, however, the reflection term is not given explicitly as a coefficient of the equation. Rather, it is a path-dependent finite-variation process whose increments are governed by the boundary geometry, the normal cone, and the Skorokhod constraint. Consequently, the numerical approximation of an RSDE cannot be reduced to the approximation of the constrained trajectory alone; it also requires a quantitative error analysis for the boundary regulator.

Approximation schemes for reflected diffusions have been studied from several viewpoints, including Euler-type, Euler--Peano,  projection, penalization, half-space, and Wong--Zakai-type approximations; see, e.g., \cite{lepingle1995euler, pettersson1997penalization, slominski2001euler, gobet2002euler, bossy2004symmetrized, aida2013wongzakai, aida2014wongzakai, zhang2014strong, wang2019approximation, ren2025wongzakai, zhang2025approximation, li2026wongzakai}. These works provide important tools for approximating constrained stochastic dynamics and for understanding the effect of the Skorokhod reflection at the discrete or pathwise approximation level. Nevertheless, most available strong convergence analyses rely essentially on globally Lipschitz or linearly growing coefficients, together with suitable regularity assumptions on the domain and the reflection mechanism. This is a restrictive framework for nonlinear constrained stochastic dynamics, where polynomial restoring forces, dissipative interactions, or state-dependent noise may naturally give rise to non-globally Lipschitz coefficients with super-linear growth. An early systematic contribution in the non-Lipschitz setting was made by Duan and Peng \cite{duan2022approximation}, who studied a penalization Euler scheme for RSDEs with non-Lipschitzian coefficients in a bounded convex domain. They established uniform mean-square convergence and further obtained convergence rates under additional polynomial-growth assumptions. More recently, modified tamed Euler and projection-type schemes for RSDEs with non-globally Lipschitz coefficients have been investigated in \cite{huang2026note, zhang2026convergence}. However, the available rates do not yet give a fixed Euler-type strong order $1/2$ for RSDEs with super-linearly growing drift and diffusion coefficients in the sense considered here. Moreover, they do not provide simultaneous order-$1/2$ estimates for the constrained state process and the boundary regulator. This leaves open the problem of designing an explicit reflected scheme that recovers the standard Euler-type strong order $1/2$ under super-linear growth and, at the same time, provides a quantitative order-$1/2$ approximation of the boundary regulator.

We address this problem by constructing a reflected explicit approximation that combines a pathwise Skorokhod correction with a
coupled taming of the drift and diffusion coefficients. The taming idea is in the spirit of tamed and other explicit stabilized
schemes developed for non-globally Lipschitz SDEs; see, e.g., \cite{hutzenthaler2012strong, wang2013tamed, tretyakov2013fundamental, sabanis2016euler, chen2019meansquare, wang2024weak}. However, in the reflected setting the taming cannot be designed only for the unconstrained SDE part. It has to be compatible with the Skorokhod correction and, at the same time, respect the coupled drift-diffusion structure used in the coercivity analysis of the RSDE. Therefore, for stepsize $h > 0$ and $x \in \R^{d}$, we introduce the common taming factor
\begin{equation*}
      \lambda_h(x)
      =
      \frac{1}{1+h^{1/2}\big(|b(x)|+\|\sigma(x)\|_{\mathrm{HS}}^{2}\big)}
\end{equation*}
and define $b_{h}(x) = \lambda_{h}(x)b(x), \sigma_{h}(x) = \lambda_{h}(x)^{1/2}\sigma(x)$. The advantage of using this common factor becomes clear from the following identity: for a fixed interior reference point $x_{\ast} \in D$ and every constant $c > 0$,
\begin{equation*}
      2\big\langle x-x_{\ast},b_{h}(x)\big\rangle
      +
      c\|\sigma_{h}(x)\|_{\mathrm{HS}}^{2}
      =
      \lambda_{h}(x)\big(2\big\langle x-x_\ast,b(x)\big\rangle
      +
      c\|\sigma(x)\|_{\mathrm{HS}}^{2}\big).
\end{equation*}
%Hence the whole drift-diffusion coercivity expression is scaled by the same factor, rather than being modified term by term. This is the structural reason for using a coupled taming, and it allows the coercivity estimates for the original coefficients to be transferred directly to the tamed coefficients. 
Thus the drift-diffusion coercivity estimate for the original coefficients is inherited directly by the tamed coefficients. Finally, by freezing $b_h$ and $\sigma_h$ at the left endpoints of the time grid and applying the Skorokhod correction, we define the coupled tamed Euler--Peano approximation $(X_h,K_h)$ by
\begin{equation}\label{eq:coupled-tamed-euler-peano-intro}
      X_{h}(t)
      =
      x_{0}
      +
      \int_{0}^{t} b_{h}\big(\overline{X}_{h}(s)\big)\,ds
      +
      \int_{0}^{t} \sigma_{h}\big(\overline{X}_{h}(s)\big)\,dW(s)
      +
      K_{h}(t), \quad t \in [0,T],
\end{equation}
where $\overline{X}_{h}$ denotes the piecewise constant interpolation of $X_{h}$ at the left endpoints.

The theoretical analysis of this approximation starts with the continuous reflected equation. Under the coupled monotonicity condition and polynomial local Lipschitz assumptions, which allow both the drift and diffusion coefficients to grow super-linearly, we first establish the well-posedness of the RSDE \eqref{eq:reflected-sde-intro} and derive suitable uniform moment bounds for the solution. We then prove uniform moment estimates for the coupled tamed Euler--Peano approximation and derive quantitative strong error bounds for both the state process and the boundary regulator. More precisely, for the admissible range of $p \geq 2$, there exists a constant $C > 0$, independent of the stepsize $h$, such that
\begin{equation*}
      \bigg(\E\left[\sup_{0\leq t\leq T}
      |X(t)-X_h(t)|^{p}\right]\bigg)^{1/p}
      +
      \bigg(\E\left[\sup_{0\leq t\leq T}
      |K(t)-K_h(t)|^{p}\right]\bigg)^{1/p}
      \leq
      C h^{1/2}.
\end{equation*}
%This estimate shows that the proposed explicit reflected scheme recovers the standard Euler-type strong order $1/2$ in the presence of super-linearly growing drift and diffusion coefficients, simultaneously for the constrained state process and the boundary regulator. To the best of our knowledge, this is the first result establishing such an order-$1/2$ estimate for both components in this setting, and it extends taming ideas for non-globally Lipschitz SDEs to the reflected case through a coupled drift--diffusion construction.
This estimate shows that the proposed explicit reflected scheme recovers the standard Euler-type strong order $1/2$ in the presence of super-linearly growing drift and diffusion coefficients. In this sense, the result reaches the usual optimal benchmark for Euler-type strong approximations driven by Brownian noise (see, e.g., \cite{kloeden1992numerical}), while simultaneously providing an order-$1/2$ approximation of the boundary regulator. This is in contrast with the existing non-globally Lipschitz reflected schemes, where the available rates either are lower than this benchmark, are obtained under different structural restrictions, or do not include a matching order estimate for the Skorokhod reflection term; see, e.g., \cite{duan2022approximation, huang2026note, zhang2026convergence}. To the best of our knowledge, this is the first result establishing such an Euler-type order-$1/2$ estimate for both the constrained state process and the boundary regulator in the reflected super-linear setting.

The rest of the paper is organized as follows. Section~\ref{sec:prelim} collects the geometric preliminaries for convex domains and the Skorokhod problem, states the standing assumptions, and proves the well-posedness and moment bounds of the RSDE. The numerical scheme is introduced in Section~\ref{sec:scheme}, where we define the coupled tamed Euler--Peano approximation and establish
uniform moment estimates for the numerical solution. Section~\ref{sec:convergence} is devoted to the strong convergence analysis, including error estimates for both the state process and the boundary regulator. Finally, Section~\ref{set:experiments} presents numerical experiments illustrating the constraint-preserving property of the method and supporting the theoretical convergence rate.

\section{Geometric preliminaries and assumptions}\label{sec:prelim}

This section collects the geometric and analytic ingredients needed in the sequel. We first introduce some notation. Let $\langle \cdot,\cdot \rangle$ and $|\cdot|$ denote the Euclidean inner product and norm in $\R^{d}$, respectively. For a matrix
$A = (a_{ij}) \in \mathbb{R}^{d \times m}$, we write $A^{\top}$ for its transpose and denote its Hilbert--Schmidt norm by
\begin{equation*}
      \|A\|_{\mathrm{HS}}
      :=
      \bigg(\sum_{i=1}^{d}\sum_{j=1}^{m}|a_{ij}|^{2}\bigg)^{1/2}
      =
      \big(\operatorname{Tr}(AA^{\top})\big)^{1/2}.
\end{equation*}
For $x \in \R^{d}$ and $r > 0$, we set $B(x,r) := \big\{y \in \R^{d}: |y-x| < r\big\}$. Throughout the paper, $C$ denotes a generic positive constant whose value may change from line to line. The constants $C_p$, $C_{p,T}$ and similar quantities may depend on the indicated parameters and on the fixed structural constants of the problem, such as the dimension, the domain constants, the coefficients, and the reference point $x_{\ast}$, but they are independent of discretization parameters, truncation levels, stopping levels, and time-grid indices unless otherwise stated.

\subsection{Convex domains and the Skorokhod problem}
\label{subsect:convexdomains}

Throughout the paper, the reflecting domain is assumed to satisfy the following convexity condition.

\begin{assumption}\label{ass:domain}
      Let $D \subset \R^{d}$ be a nonempty open convex domain with $\partial D \neq \varnothing$.
\end{assumption}

Since $D$ is nonempty, we fix a reference point $x_\ast\in D$ for later use. For $x \in \partial{D}$ and $r > 0$, we set $N_{x,r} := \big\{\mathbf{n} \in \R^{d}: |\mathbf{n}| = 1,\,B(x-r\mathbf{n},r) \cap D = \varnothing\big\}$, and define
\begin{equation*}
      N_{x}
      :=
      \bigcup_{r>0}N_{x,r}.
\end{equation*}
The vectors in $N_x$ are called inward unit normal vectors to $\partial D$ at $x$. Since $D$ is convex, the above definition is equivalent to
\begin{equation}\label{eq:normal-cone-characterization}
      N_{x}
      =
      \big\{
      \mathbf{n} \in\R^{d}:
      |\mathbf{n}|=1,\,
      \langle y-x,\mathbf{n}\rangle\geq0
      \text{ for every }y\in\overline{D}
      \big\}.
\end{equation}

For a continuous function $\phi \colon [0,T] \to \R^{d}$ of bounded variation, we denote by $|\phi|_{[s,t]}$ its total variation over the interval $[s,t]$, and by $d|\phi|$ the corresponding variation measure. Let $w\in C\big([0,T];\R^{d}\big)$ with $w(0) \in \overline{D}$. A pair $(\xi,\phi) \in C\big([0,T];\overline{D}\big) \times C\big([0,T];\R^{d}\big)$ is called a solution of the Skorokhod problem associated with $w$ if the following conditions are satisfied:
\begin{enumerate}
      \item $\xi(t)=w(t)+\phi(t)$ for every $t\in[0,T]$;

      \item $\phi(0)=0$ and $\phi$ is a continuous function of bounded variation;

      \item there exists a measurable function $\mathbf{n} \colon [0,T] \to \R^{d}$ such that
            \begin{equation*}
                  \phi(t)
                  =
                  \int_{0}^{t} \mathbf{n}(s)\,d|\phi|(s),
                  \quad \mathbf{n}(s) \in N_{\xi(s)}
                  \quad d|\phi|\text{-a.e.};
            \end{equation*}

      \item the variation measure of $\phi$ is supported on the boundary, namely,
            \begin{equation*}
                  |\phi|_{[0,t]}
                  =
                  \int_{0}^{t}
                  \mathbf{1}_{\{\xi(s)\in\partial D\}}\,d|\phi|(s),
                  \quad t \in [0,T].
            \end{equation*}
\end{enumerate}

For a convex domain, the Skorokhod problem admits a unique solution for every continuous driving path $w\in C\big([0,T];\R^{d}\big)$ satisfying $w(0) \in \overline{D}$; see, e.g., \cite{tanaka1979stochastic,aida2014wongzakai}. We denote the corresponding Skorokhod map by $\xi = \Gamma(w)$. Moreover, the convexity of $D$ yields the following standard monotonicity properties of the reflection term; see, e.g., \cite{tanaka1979stochastic, lions1984stochastic}.

\begin{lemma}\label{lem:reflection-monotonicity}
      Let $(\xi,\phi)$ be the solution of the Skorokhod problem associated with a path $w\in C\big([0,T];\R^{d}\big)$ satisfying $w(0) \in \overline{D}$. Then for every $y \in \overline{D}$ and every $t \in [0,T]$,
      \begin{equation}\label{eq:reflection-one-path}
            \int_{0}^{t}
            \langle \xi(s)-y,d\phi(s)\rangle
            \leq0.
      \end{equation}
      For $i = 1,2$, let $(\xi_i,\phi_i)$ be the solution of the Skorokhod problem associated with two continuous paths $w_{i} \in C\big([0,T];\R^{d}\big)$ satisfying $w_{i}(0) \in \overline{D}$. Then
      \begin{equation}\label{eq:reflection-two-paths}
            \int_{0}^{t}\big\langle
            \xi_1(s)-\xi_2(s), d\phi_1(s)-d\phi_2(s) \big\rangle
            \leq0,
            \quad t \in [0,T].
      \end{equation}
\end{lemma}

\begin{proof}
      By the definition of the Skorokhod problem, we have
      \begin{equation*}
            d\phi(s)
            =
            \mathbf{n}(s)\,d|\phi|(s),
            \quad \mathbf{n}(s)\in N_{\xi(s)}
            \quad d|\phi|\text{-a.e.}
      \end{equation*}
      Since the variation measure $d|\phi|$ is supported on the boundary, one has $\xi(s)\in\partial D$ for $d|\phi|$-almost every $s$. Hence, by \eqref{eq:normal-cone-characterization}, we have that for every $y \in \overline{D}$,
      \begin{equation*}
            \langle \xi(s)-y,\mathbf{n}(s)\rangle
            \leq0
            \quad
            d|\phi|\text{-a.e.},
      \end{equation*}
      which gives $\langle \xi(s)-y,d\phi(s)\rangle = \langle \xi(s)-y,\mathbf{n}(s)\rangle \,d|\phi|(s) \leq 0$. Integrating over $[0,t]$ proves \eqref{eq:reflection-one-path}. For the second assertion, write
      \begin{equation*}
            d\phi_i(s)
            =
            \mathbf{n}_i(s)\,d|\phi_i|(s),
            \quad
            \mathbf{n}_i(s)\in N_{\xi_i(s)}
            \quad
            d|\phi_i|\text{-a.e.},
            \quad
            i=1,2.
      \end{equation*}
      Since $\xi_2(s) \in\overline{D}$, the normal-cone characterization implies
      \begin{equation*}
            \langle \xi_{1}(s) - \xi_{2}(s),
            \mathbf{n}_{1}(s) \rangle
            \leq0
            \quad
            d|\phi_1|\text{-a.e.}
      \end{equation*}
      Similarly, since $\xi_1(s) \in \overline{D}$,
      \begin{equation*}
            \langle \xi_{2}(s) - \xi_{1}(s),
            \mathbf{n}_{2}(s)
            \rangle
            \leq0
            \quad
            d|\phi_2|\text{-a.e.}.
      \end{equation*}
      It follows that
      \begin{align*}
            &~\big\langle \xi_1(s)-\xi_2(s),
            d\phi_1(s)-d\phi_2(s) \big\rangle
            \\=&~
            \langle \xi_1(s)-\xi_2(s), \mathbf{n}_{1}(s) \rangle \,d|\phi_1|(s)
            - 
            \langle \xi_1(s)-\xi_2(s),\mathbf{n}_{2}(s) \rangle \,d|\phi_2|(s)
            \leq
            0.
      \end{align*}
      Integrating over $[0,t]$ proves \eqref{eq:reflection-two-paths}.
\end{proof}

\subsection{RSDE and coefficient assumptions}\label{subsec:rsde-coefficients}
We now give the precise formulation of the RSDE and state the assumptions on the coefficients. Let $\big(\Omega, \mathcal{F}, \{\mathcal{F}_{t}\}_{t \geq 0}, \mathbb{P}\big)$ be a complete filtered probability space satisfying the usual conditions, and let $\{W(t)\}_{t \geq 0}$ be an $m$-dimensional standard Brownian motion adapted to $\{\mathcal{F}_{t}\}_{t \geq 0}$. Let $b \colon \R^{d} \to \R^{d}$, $\sigma \colon \R^{d} \to \R^{d \times m}$, and let $x_{0}\in\overline{D}$ be deterministic. Consider 
\begin{equation}\label{eq:reflected-sde}
      X(t)
      =
      x_{0}
      +
      \int_{0}^{t} b\big(X(s)\big) \,ds
      +
      \int_{0}^{t} \sigma\big(X(s)\big)\,dW(s)
      +
      K(t), \quad t \in [0,T],
\end{equation}
where $\{X(t)\}_{t \in [0,T]}$ is an $\overline{D}$-valued continuous adapted process, and $\{K(t)\}_{t \in [0,T]}$ is a continuous adapted process of bounded variation satisfying $K(0) = 0$. More precisely, there exists a progressively measurable process $\mathbf{n}_{X} \colon [0,T] \times \Omega \to \R^{d}$ such that
\begin{equation}\label{eq:exact-reflection-decomposition}
      K(t)
      =
      \int_{0}^{t} \mathbf{n}_{X}(s)\,d|K|(s),
      \quad \mathbf{n}_{X}(s) \in N_{X(s)}
      \quad d|K|\text{-a.e.},
\end{equation}
and the variation measure of $K$ is supported on the boundary:
\begin{equation}\label{eq:exact-reflection-support}
      |K|_{[0,t]}
      =
      \int_{0}^{t} \mathbf{1}_{\{X(s)\in\partial D\}} \,d|K|(s),
      \quad t\in[0,T].
\end{equation}
A pair $(X,K)$ satisfying \eqref{eq:reflected-sde}--\eqref{eq:exact-reflection-support} is called a strong solution of the RSDE.

The following conditions allow the drift and diffusion coefficients to grow super-linearly while preserving a coupled dissipative structure.

\begin{assumption}[Coefficient conditions]\label{ass:coefficients}
      The coefficients $b$ and $\sigma$ are continuous and satisfy the following conditions.

      \begin{enumerate}
            \item[(i)] There exist constants $L>0$ and $\eta > \frac{1}{2}$ such that
            \begin{equation}\label{eq:coupled-monotonicity}
                  \big\langle x-y,b(x)-b(y) \big\rangle
                  +
                  \eta \big\|\sigma(x)-\sigma(y)\big\|_{\mathrm{HS}}^{2}
                  \leq
                  L|x-y|^{2}, \quad x,y\in\R^{d}.
            \end{equation}

            \item[(ii)] There exist constants $L>0$ and $q_{b}\geq0$ such that
            \begin{equation}\label{eq:drift-polynomial-lipschitz}
                  |b(x)-b(y)|
                  \leq
                  L\big(1 + |x|^{q_{b}} + |y|^{q_{b}}\big)|x-y|,
                  \quad x,y \in \R^{d}.
            \end{equation}

            \item[(iii)] There exist constants $L>0$ and $q_{\sigma}\geq0$ such that
            \begin{equation}\label{eq:diffusion-polynomial-lipschitz}
                  \big\|\sigma(x)-\sigma(y)\big\|_{\mathrm{HS}}
                  \leq
                  L\big(1 + |x|^{q_{\sigma}} + |y|^{q_{\sigma}}\big)|x-y|,
                  \quad x,y \in \R^{d}.
            \end{equation}
      \end{enumerate}
\end{assumption}

The coupled monotonicity condition \eqref{eq:coupled-monotonicity} allows the dissipativity of the drift coefficient to compensate for the super-linear growth of the diffusion coefficient. In particular, neither $b$ nor $\sigma$ is required to be globally Lipschitz continuous. Conditions \eqref{eq:drift-polynomial-lipschitz} and \eqref{eq:diffusion-polynomial-lipschitz} imply the polynomial growth estimates
\begin{equation}\label{eq:coefficient-growth}
      |b(x)| \leq C\big(1+|x|^{q_{b}+1}\big),
      \quad
      \|\sigma(x)\|_{\mathrm{HS}} \leq C\big(1+|x|^{q_{\sigma}+1}\big),
      \quad x \in \R^{d}.
\end{equation}
For later use, we define the maximal moment index supplied by  \eqref{eq:coupled-monotonicity} as
\begin{equation}\label{eq:maximal-moment-index}
      p_{\ast} := \eta+\frac{1}{2}.
\end{equation}
The role of the index $p_\ast$ is clarified by the following coercivity estimate, which is a direct consequence of the coupled monotonicity condition.

\begin{lemma}\label{lem:high-order-coercivity}
      Suppose that Assumption~\ref{ass:coefficients} holds and let $1 \leq p < p_{\ast}$. Then there exist constants $\varepsilon_{p} > 0$ and $C_{p} > 0$ such that
      \begin{equation}\label{eq:high-order-coercivity}
            2\big\langle x-x_{\ast},b(x) \big\rangle
            +
            \big(2p-1+\varepsilon_{p}\big)\|\sigma(x)\|_{\mathrm{HS}}^{2}
            \leq
            C_{p}\big(1 + |x-x_{\ast}|^{2}\big),
            \quad x \in \R^{d}.
      \end{equation}
\end{lemma}

\begin{proof}
      Fix $p \in [1,p_{\ast})$. Since $2p-1 < 2\eta$, we may choose a constant $\theta_{p}\in(0,1)$ such that
      \begin{equation}\label{eq:theta-choice}
            2p-1 < 2\eta\big(1 - \theta_{p}\big).
      \end{equation}
      For arbitrary matrices $A,B\in\R^{d\times m}$, Young's inequality yields
      \begin{equation}\label{eq:diffusion-lower-bound}
            \|A\|_{\mathrm{HS}}^{2}
            =
            \|A-B+B\|_{\mathrm{HS}}^{2}
            \leq
            \frac{1}{1-\theta_{p}}\|A-B\|_{\mathrm{HS}}^{2}
            +
            \frac{1}{\theta_{p}}\|B\|_{\mathrm{HS}}^{2}.
      \end{equation}
      Taking $y = x_{\ast}$ in \eqref{eq:coupled-monotonicity} and applying \eqref{eq:diffusion-lower-bound} with $A = \sigma(x)$ and $B = \sigma(x_{\ast})$, we obtain
      \begin{align*}
            2\big\langle x-x_{\ast},b(x)-b(x_{\ast}) \big\rangle
            +
            2\eta \big(1-\theta_{p}\big)\|\sigma(x)\|_{\mathrm{HS}}^{2}
            \leq
            2L|x-x_{\ast}|^{2}
            +
            \frac{2\eta\big(1-\theta_{p}\big)}{\theta_{p}}
            \|\sigma(x_{\ast})\|_{\mathrm{HS}}^{2}.
      \end{align*}
      Moreover, Young's inequality implies $2\big\langle x-x_{\ast},b(x_{\ast}) \big\rangle \leq |x-x_{\ast}|^{2} + |b(x_{\ast})|^{2}$. It follows that
      \begin{equation*}
            2\big\langle x-x_{\ast},b(x) \big\rangle
            +
            2\eta\big(1-\theta_{p}\big) \|\sigma(x)\|_{\mathrm{HS}}^{2}
            \leq
            C_{p}\big(1 + |x-x_{\ast}|^{2}\big)
      \end{equation*}
      with $$C_{p} := 1+2L + |b(x_{\ast})|^{2}+\frac{2\eta\big(1-\theta_{p}\big)}{\theta_{p}} \|\sigma(x_{\ast})\|_{\mathrm{HS}}^{2}.$$
      Define $\varepsilon_{p} := 2\eta\big(1 - \theta_{p}\big) - (2p-1)$. By \eqref{eq:theta-choice}, one has $\varepsilon_{p}>0$, and thus \eqref{eq:high-order-coercivity}.
\end{proof}

\subsection{Well-posedness and moment estimates}\label{subsec:well-posedness}

We next establish the well-posedness of the RSDE \eqref{eq:reflected-sde}--\eqref{eq:exact-reflection-support} and derive
the moment estimates needed in the numerical analysis. Since both the drift and diffusion coefficients may grow super-linearly, the proof is based on localization and a consistency argument for truncated equations. In what follows, we use the Lyapunov function
\begin{equation}\label{eq:lyapunov-function}
      V(x)
      :=
      1+|x-x_{\ast}|^{2},
      \quad
      x\in\R^{d}.
\end{equation}

\begin{theorem}\label{thm:well-posedness}%[Well-posedness and moment estimates]
      Suppose that Assumptions~\ref{ass:domain} and \ref{ass:coefficients} hold. Then the RSDE \eqref{eq:reflected-sde}--%
      \eqref{eq:exact-reflection-support} admits a unique global strong solution $(X,K)$. Moreover, for every $1 \leq p < p_{\ast}$ and every $T > 0$, there exists a constant $C_{p,T} > 0$ such that
      \begin{equation}\label{eq:exact-solution-sup-moment}
            \E\left[\sup_{0 \leq t \leq T} V^{p} \big(X(t)\big)\right]
            \leq
            C_{p,T} V^{p}(x_{0}).
      \end{equation}
      Consequently,
      \begin{equation}\label{eq:exact-solution-absolute-moment}
            \E\left[\sup_{0 \leq t \leq T}|X(t)|^{2p}\right]
            \leq
            C_{p,T}\big(1 + |x_{0}|^{2p}\big).
      \end{equation}
\end{theorem}

\begin{proof}
      We divide the proof into five steps.

      \medskip

      \noindent
      \textbf{Step 1. Construction of globally Lipschitz truncated equations.}
      Choose an integer $n_{0}$ such that $n_{0} > |x_{0}| \vee |x_{\ast}|$. For every integer $n \geq n_{0}$, we define
      \begin{equation*}
            \Pi_{n}(x)
            :=
            \begin{cases}
                  x, & |x|\leq n,
                  \\[2mm]\displaystyle
                  \frac{n}{|x|}x,& |x|>n.
            \end{cases}
      \end{equation*}
      The mapping $\Pi_{n}$ is the metric projection of $\R^{d}$ onto the closed ball $\overline{B}_{n} := \big\{x \in \R^{d} : |x| \leq n\big\}$. Since $B_n$ is closed and convex, the metric projection $\Pi_n=P_{B_n}$ is firmly nonexpansive, and hence nonexpansive; see, e.g., \cite[Proposition~4.8]{bauschke2011convex}. Therefore,
      \begin{equation}\label{eq:projection-nonexpansive}
            |\Pi_{n}(x)-\Pi_{n}(y)| \leq |x-y|,
            \quad x,y \in \R^{d}.
      \end{equation}
      Define the truncated coefficients $b_{n}(x) := b\big(\Pi_{n}(x)\big)$ and $\sigma_{n}(x) := \sigma\big(\Pi_{n}(x)\big)$ for all $x \in \R^{d}$. It follows immediately that $b_{n}(x) = b(x), \sigma_{n}(x) = \sigma(x)$ for all $x \in \R^{d}$ with $|x| \leq n$. By \eqref{eq:drift-polynomial-lipschitz} and \eqref{eq:projection-nonexpansive}, we obtain
      \begin{align*}
            |b_{n}(x)-b_{n}(y)|
            =
            \left|b\big(\Pi_{n}(x)\big) - b\big(\Pi_{n}(y)\big)\right|
            \leq
            L\big(1 + 2n^{q_{b}}\big)|\Pi_{n}(x)-\Pi_{n}(y)|
            \leq
            L\big(1 + 2n^{q_{b}}\big)|x-y|,
      \end{align*}
      and similarly $\|\sigma_{n}(x)-\sigma_{n}(y)\|_{\mathrm{HS}} \leq L\big(1 + 2n^{q_{\sigma}}\big)|x-y|$, i.e., both $b_{n}$ and $\sigma_{n}$ are globally Lipschitz continuous. For each $n\geq n_{0}$, consider the truncated RSDE
      \begin{equation}\label{eq:truncated-reflected-sde}
            X^{n}(t)
            =
            x_{0}
            +
            \int_{0}^{t} b_{n}\big(X^{n}(s)\big)\,ds
            +
            \int_{0}^{t} \sigma_{n}\big(X^{n}(s)\big)\,dW(s)
            +
            K^{n}(t), \quad t \geq 0.
      \end{equation}
      Here, $X^{n}$ takes values in $\overline{D}$, and $K^{n}$ is a continuous adapted process of bounded variation satisfying
      \begin{equation}\label{eq:truncated-reflection-decomposition}
            K^{n}(t)
            =
            \int_{0}^{t} \mathbf{n}^{n}(s)\,d|K^{n}|(s),
            \quad \mathbf{n}^{n}(s)\in N_{X^{n}(s)}
            \quad d|K^{n}|\text{-a.e.},
      \end{equation}
      together with
      \begin{equation}\label{eq:truncated-reflection-support}
            |K^{n}|_{[0,t]}
            =
            \int_{0}^{t} \mathbf{1}_{\{X^{n}(s)\in\partial D\}} \,d|K^{n}|(s).
      \end{equation}
      Since $b_{n}$ and $\sigma_{n}$ are globally Lipschitz continuous, the classical well-posedness theory for the RSDEs on convex domains yields a unique global strong solution $(X^{n},K^{n})$ to \eqref{eq:truncated-reflected-sde}--%
      \eqref{eq:truncated-reflection-support}; see, e.g., \cite{tanaka1979stochastic}. 

      \medskip

      \noindent\textbf{Step 2. Consistency of the truncated solutions.} Let $m > n \geq n_{0}$ and define
      \begin{equation*}
            \tau_{n}^{n}
            :=
            \inf\big\{t \geq 0: |X^{n}(t)| \geq n \big\},
            \quad
            \tau_{n}^{m}
            :=
            \inf\big\{t \geq 0: |X^{m}(t)| \geq n\big\},
      \end{equation*}
      where $\inf\varnothing := \infty$, and set $\rho_{n}^{n,m} := \tau_{n}^{n} \wedge \tau_{n}^{m}$. For $t < \rho_{n}^{n,m}$, both $X^{n}(t)$ and $X^{m}(t)$ belong to $B(0,n)$. Hence,
      \begin{gather*}
            b_{n}\big(X^{n}(t)\big)
            =
            b\big(X^{n}(t)\big),
            \quad
            \sigma_{n}\big(X^{n}(t)\big)
            =
            \sigma\big(X^{n}(t)\big),
            \\
            b_{m}\big(X^{m}(t)\big)
            =
            b\big(X^{m}(t)\big),
            \quad
            \sigma_{m}\big(X^{m}(t)\big)
            =
            \sigma\big(X^{m}(t)\big).
      \end{gather*}
      Applying It\^{o}'s formula to $|X^{n}\big(t\wedge\rho_{n}^{n,m}\big) - X^{m}\big(t\wedge\rho_{n}^{n,m}\big)|^{2}$ and using \eqref{eq:reflection-two-paths}, we obtain
      \begin{align*}
            &~\E\big[\left|X^{n}\big(t\wedge\rho_{n}^{n,m}\big)
            - X^{m}\big(t\wedge\rho_{n}^{n,m}\big)\right|^{2}\big]
            \\\leq&~
            \E\bigg[\int_{0}^{t\wedge\rho_{n}^{n,m}}
            \big(2\big\langle X^{n}(s)-X^{m}(s),
            b\big(X^{n}(s)\big)-b\big(X^{m}(s)\big) \big\rangle
            \\&~+
            \big\|\sigma\big(X^{n}(s)\big) - \sigma\big(X^{m}(s)\big)
            \big\|_{\mathrm{HS}}^{2}\big)\,ds\bigg]
            \\\leq&~
            2L\int_{0}^{t}\E\big[\left|
            X^{n}\big(s\wedge\rho_{n}^{n,m}\big)
            - X^{m}\big(s\wedge\rho_{n}^{n,m}\big)\right|^{2}\big]\,ds
      \end{align*}
      due to \eqref{eq:coupled-monotonicity}. Gronwall's inequality yields
      \begin{equation}\label{eq:truncated-solutions-consistency}
            X^{n}(t) = X^{m}(t),
            \quad 0 \leq t \leq \rho_{n}^{n,m},
            \quad \mathbb{P}\text{-a.s.}
      \end{equation}
      By the continuity of $X^{n}$ and $X^{m}$, \eqref{eq:truncated-solutions-consistency} implies
      \begin{equation*}
            \tau_{n}^{n} = \tau_{n}^{m},
            \quad
            \mathbb{P}\text{-a.s.}
      \end{equation*}
      We denote their common value by $\tau_{n}$. Subtracting the two reflected equations and using
      \eqref{eq:truncated-solutions-consistency}, we also obtain
      \begin{equation}\label{eq:truncated-reflections-consistency}
            K^{n}(t)
            =
            K^{m}(t),
            \quad
            0\leq t\leq\tau_{n},
            \quad
            \mathbb{P}\text{-a.s.}
      \end{equation}
      The consistency relations may be assumed to hold simultaneously for all integers $m > n \geq n_{0}$ on an event of probability one. Moreover, $\tau_{n} \leq \tau_{n+1}$. Define $\tau_{\infty} := \lim_{n \to \infty} \tau_{n}$. For every  $t < \tau_{\infty}$, choose $n$ sufficiently large so that $t < \tau_{n}$, and define 
      \begin{equation}\label{eq:maximal-local-solution}
            X(t) := X^{n}(t),
            \quad
            K(t) := K^{n}(t).
      \end{equation}
      By \eqref{eq:truncated-solutions-consistency} and \eqref{eq:truncated-reflections-consistency}, this definition is
      independent of the choice of $n$. The pair $(X,K)$ is therefore a maximal local strong solution of
      \eqref{eq:reflected-sde}--\eqref{eq:exact-reflection-support} on $[0,\tau_{\infty})$.

      \medskip

      \noindent\textbf{Step 3. Stopped moment estimates and nonexplosion.}
      Fix $1 \leq P < p_{\ast}$. Set $Z^{n}(t) := X^{n}(t)-x_{\ast}$. Since $b_{n}\big(X^{n}(s)\big) = b\big(X^{n}(s)\big),          \sigma_{n}\big(X^{n}(s)\big) = \sigma\big(X^{n}(s)\big)$ for $s \leq \tau_{n}$, It\^{o}'s formula applied to
      $V^{P}\big(X^{n}(t \wedge \tau_{n})\big)$ gives
      \begin{align}\label{eq:ito-stopped-exact-solution}
            V^{P}\big(X^{n}(t\wedge\tau_{n})\big) 
            =&~
            V^{P}(x_{0})
            +
            P\int_{0}^{t\wedge\tau_{n}}V^{P-1}\big(X^{n}(s)\big)
            \big(2 \big\langle Z^{n}(s),b\big(X^{n}(s)\big) \big\rangle
            +
            \big\|\sigma\big(X^{n}(s)\big)\big\|_{\mathrm{HS}}^{2}\big) \,ds \nonumber
            \\&~+
            2P(P-1) \int_{0}^{t\wedge\tau_{n}} V^{P-2}\big(X^{n}(s)\big)
            \left|\sigma^{\top}\big(X^{n}(s)\big) Z^{n}(s) \right|^{2} \,ds \nonumber
            \\&~+
            2P\int_{0}^{t\wedge\tau_{n}}V^{P-1}\big(X^{n}(s)\big)
            \big\langle Z^{n}(s),\sigma\big(X^{n}(s)\big)\,dW(s) \big\rangle \nonumber
            \\&~+
            2P\int_{0}^{t\wedge\tau_{n}} V^{P-1}\big(X^{n}(s)\big)
            \big\langle Z^{n}(s), dK^{n}(s) \big\rangle.
      \end{align}
      By \eqref{eq:reflection-one-path} with $y = x_{\ast}$, we have
      \begin{equation}\label{eq:exact-reflection-nonpositive}
            \int_{0}^{t\wedge\tau_{n}} V^{P-1}\big(X^{n}(s)\big)
            \big\langle Z^{n}(s),dK^{n}(s) \big\rangle
            \leq
            0.
      \end{equation}
      Moreover,
      \begin{equation*}
            \left|\sigma^{\top}(x)(x-x_{\ast})\right|^{2}
            \leq
            |x-x_{\ast}|^{2}\|\sigma(x)\|_{\mathrm{HS}}^{2}
            \leq
            V(x) \|\sigma(x)\|_{\mathrm{HS}}^{2}.
      \end{equation*}
      Therefore, the sum of the two finite-variation terms in \eqref{eq:ito-stopped-exact-solution} is bounded above by
      \begin{align*}
            P\int_{0}^{t\wedge\tau_{n}}V^{P-1}\big(X^{n}(s)\big)
            \big(2 \big\langle Z^{n}(s),b\big(X^{n}(s)\big) \big\rangle
            +
            (2P-1)\big\|\sigma\big(X^{n}(s)\big)\big\|_{\mathrm{HS}}^{2}\big)\,ds.
      \end{align*}
      By Lemma~\ref{lem:high-order-coercivity}, we have
      \begin{align*}
            2 \langle x-x_{\ast},b(x) \rangle
            +
            (2P-1) \|\sigma(x)\|_{\mathrm{HS}}^{2}
            \leq
            C_{P}V(x) - \varepsilon_{P}\|\sigma(x)\|_{\mathrm{HS}}^{2}
            \leq
            C_{P}V(x).
      \end{align*}
      Taking expectations in \eqref{eq:ito-stopped-exact-solution}, using \eqref{eq:exact-reflection-nonpositive}, and noting that the stopped stochastic integral has zero expectation, we obtain
      \begin{equation*}
            \E \big[V^{P}\big(X^{n}(t\wedge\tau_{n})\big)\big]
            \leq
            V^{P}(x_{0})
            +
            C_{P} \int_{0}^{t} \E\big[
            V^{P}\big(X^{n}(s\wedge\tau_{n})\big)\big] \,ds.
      \end{equation*}
      Gronwall's inequality yields
      \begin{equation}\label{eq:stopped-exact-moment}
            \sup_{0\leq t\leq T} \E
            \big[V^{P}\big(X^{n}(t\wedge\tau_{n})\big)\big]
            \leq
            C_{P,T}V^{P}(x_{0}),
      \end{equation}
      where $C_{P,T}$ is independent of $n$. On the event $\{\tau_{n}\leq T\}$, the continuity of $X^{n}$ gives $|X^{n}(\tau_{n})| =             n$. Hence, $V\big(X^{n}(\tau_{n})\big) = 1 + |X^{n}(\tau_{n})-x_{\ast}|^{2} \geq 1 + (n-|x_{\ast}|)^{2}$. It follows from \eqref{eq:stopped-exact-moment} that
      \begin{equation}\label{eq:truncation-exit-probability}
            \mathbb{P}(\tau_{n}\leq T)
            \leq
            \frac{C_{P,T}V^{P}(x_{0})}
            {\big(1 + (n-|x_{\ast}|)^{2}\big)^{P}}.
      \end{equation}
      Letting $n \to \infty$ in \eqref{eq:truncation-exit-probability} gives $\mathbb{P}\big(\tau_{\infty} \leq T\big) = 0$. Since $T > 0$ is arbitrary, we obtain $\mathbb{P}\big( \tau_{\infty} = \infty\big) = 1$. Thus, the maximal local solution $(X,K)$ constructed in Step~2 is global.

      \medskip

      \noindent\textbf{Step 4. Uniform-in-time moment estimates.} Fix $1 \leq p < p_{\ast}$. Choose an exponent $P$ such that
      \begin{equation}\label{eq:intermediate-moment-index}
            p < P < p_{\ast}.
      \end{equation}
      For $R > V(x_{0})$, define $\vartheta_{R} := \inf\big\{t \geq 0: V\big(X(t)\big) \geq R\big\}$. Repeating the stopped It\^{o} argument from Step~3 gives
      \begin{equation}\label{eq:exact-solution-stopped-P-moment}
            \sup_{0\leq t\leq T}\E\big[
            V^{P}\big(X(t\wedge\vartheta_{R})\big)\big]
            \leq
            C_{P,T}V^{P}(x_{0}),
      \end{equation}
      where the constant is independent of $R$. Define
      \begin{equation*}
            M_{T}
            :=
            \sup_{0\leq t\leq T}
            V\big(X(t)\big).
      \end{equation*}
      By continuity, on the event $\{M_{T}\geq R\}$ one has $V\big(X(\vartheta_{R})\big) = R$. Hence, \eqref{eq:exact-solution-stopped-P-moment} implies
      \begin{equation}\label{eq:exact-solution-tail-estimate}
            \mathbb{P}\big(M_{T}\geq R\big)
            \leq
            \frac{C_{P,T}V^{P}(x_{0})}{R^{P}}.
      \end{equation}
      Set $A_{P,T} := C_{P,T}V^{P}(x_{0})$. Using the tail-integral representation for non-negative random variables (see, e.g., \cite[Exercise~2.2.7]{durrett2019probability}) and \eqref{eq:exact-solution-tail-estimate}, we obtain
      \begin{align*}
            \E\big[M_{T}^{p}\big]
            =&~
            p\int_{0}^{\infty} R^{p-1}\mathbb{P}(M_{T}\geq R) \,dR
            \\\leq&~
            p\int_{0}^{A_{P,T}^{1/P}}R^{p-1}\,dR
            +
            pA_{P,T}\int_{A_{P,T}^{1/P}}^{\infty}R^{p-P-1}\,dR
            \\\leq&~
            C_{p,P,T}A_{P,T}^{p/P}
            \\\leq&~
            C_{p,P,T}V^{p}(x_{0}).
      \end{align*}
      This proves \eqref{eq:exact-solution-sup-moment}. Finally, since $|x|^{2p} \leq C_{p}\big(1 + |x-x_{\ast}|^{2p}\big) \leq C_{p}V^{p}(x)$ and $V^{p}(x_{0}) \leq C_{p}\big(1 + |x_{0}|^{2p}\big)$, estimate \eqref{eq:exact-solution-absolute-moment} follows.

      \medskip

      \noindent\textbf{Step 5. Pathwise uniqueness.} Let $(X,K)$ and $(Y,\widetilde{K})$ be two global strong solutions driven by the same Brownian motion and with the same initial value. For $N \geq 1$, define
      \begin{equation*}
            \theta_{N}
            :=
            \inf\big\{t \geq 0:|X(t)| \vee |Y(t)| \geq N\big\}.
      \end{equation*}
      Applying It\^{o}'s formula to $|X(t\wedge\theta_{N}) - Y(t\wedge\theta_{N})|^{2}$ and using \eqref{eq:reflection-two-paths}, we obtain
      \begin{align*}
            &~\E\big[|X(t\wedge\theta_{N})-Y(t\wedge\theta_{N})|^{2}\big]
            \\\leq&~
            \E\bigg[\int_{0}^{t\wedge\theta_{N}}
            \big(2 \big\langle X(s)-Y(s),
            b\big(X(s)\big)-b\big(Y(s)\big) \big\rangle
            +
            \big\|\sigma\big(X(s)\big) 
            - \sigma\big(Y(s)\big)\big\|_{\mathrm{HS}}^{2}\big)\,ds\bigg]
            \\\leq&~
            2L\int_{0}^{t}\E\big[
            |X(s\wedge\theta_{N}) - Y(s\wedge\theta_{N})|^{2}\big]\,ds.
      \end{align*}
      Gronwall's inequality gives $\E\big[|X(t\wedge\theta_{N}) - Y(t\wedge\theta_{N})|^{2}\big] = 0$. Letting $N \to \infty$ and using the nonexplosion of both solutions, we conclude that $\mathbb{P}(X(t) = Y(t), t \geq 0) = 1$. Subtracting the two reflected equations then yields $\mathbb{P}(K(t) = \widetilde{K}(t), t \geq 0) = 1$. Thus, pathwise uniqueness holds.
\end{proof}

\section{The coupled tamed Euler--Peano scheme and uniform moment estimates}\label{sec:scheme}

This section introduces the coupled tamed Euler--Peano scheme and establishes uniform moment estimates for the numerical solution. We first define the tamed coefficients and the reflected approximation, then derive algebraic properties of the taming mechanism, local increment estimates, and uniform moment bounds. To describe the admissible moment range, we define 
\begin{equation}\label{eq:auxiliary-growth-indices}
      \gamma_{b} := q_{b}+1,
      \quad
      \gamma_{\sigma} := q_{\sigma}+1,
      \quad
      \gamma := \max\big\{\gamma_{b},\gamma_{\sigma}\big\},
      \quad
      \widehat{\gamma} := \max\big\{\gamma_{b},2\gamma_{\sigma}\big\},
\end{equation}
and set
\begin{equation}\label{eq:moment-consumption-index}
      \Lambda\big(q_{b},q_{\sigma}\big)
      :=
      \gamma + \widehat{\gamma}
      =
      \max\big\{q_{b} + 1, q_{\sigma} + 1\big\}
      +
      \max\big\{q_{b} + 1, 2q_{\sigma} + 2\big\}.
\end{equation}
The index $\Lambda(q_{b},q_{\sigma})$ measures the amount of moment integrability required to control simultaneously the freezing defects and the coupled taming defects.

\subsection{The coupled tamed Euler--Peano approximation}\label{subsec:coupled-tamed-approximation}
We first define the temporal grid, the coupled tamed coefficients, and the reflected Euler--Peano approximation. Let $N \in \mathbb{N}$ and consider the uniform temporal partition
\begin{equation*}
      0 = t_{0} < t_{1} < \cdots < t_{N} = T,
      \quad t_{k} = kh,
      \quad h = \frac{T}{N}.
\end{equation*}
Throughout the paper, we assume that $h \in (0,1]$. Define the left-endpoint projection $\kappa_{h}\colon[0,T]\to[0,T]$ by
\begin{equation}\label{eq:left-endpoint-projection}
      \kappa_{h}(t) := t_{k},
      \quad t \in [t_{k},t_{k+1}), k = 0,1,\cdots,N-1.
\end{equation}
% and {\color{red}set $\kappa_{h}(T) := t_{N-1}$}. 
To preserve the coupled coercivity structure between the drift and diffusion coefficients, define
\begin{equation*}
      G(x)
      :=
      |b(x)| + \|\sigma(x)\|_{\mathrm{HS}}^{2},
      \quad
      \lambda_{h}(x)
      :=
      \frac{1}{1+h^{1/2}G(x)}, \quad x \in \R^{d}.
\end{equation*}
The coupled tamed coefficients are defined by
\begin{equation}\label{eq:coupled-tamed-coefficients}
      b_{h}(x) := \lambda_{h}(x)b(x),
      \quad
      \sigma_{h}(x) := \lambda_{h}^{1/2}(x)\sigma(x),
      \quad x \in \R^{d}.
\end{equation}
The same taming factor is applied to the drift and to the squared diffusion coefficient. Consequently, we have that for every $c \geq 0$,
\begin{align}\label{eq:formal-coercivity-preservation}
      2\big\langle x-x_{\ast},b_{h}(x) \big\rangle
      +
      c\|\sigma_{h}(x)\|_{\mathrm{HS}}^{2}
      =
      \lambda_{h}(x) \big(2\big\langle x-x_{\ast},b(x) \big\rangle
      +
      c\|\sigma(x)\|_{\mathrm{HS}}^{2}\big).
\end{align}
This identity is the main reason for using the coupled taming mechanism in \eqref{eq:coupled-tamed-coefficients}. For a continuous adapted process $\{X_{h}(t)\}_{t \in [0,T]}$, define its piecewise constant left-endpoint interpolation by
\begin{equation}\label{eq:frozen-numerical-process}
      \overline{X}_{h}(t)
      :=
      X_{h}\big(\kappa_{h}(t)\big), \quad t \in [0,T].
\end{equation}
With these tamed coefficients, the reflected approximation is defined as follows.

\begin{definition}[Coupled tamed Euler--Peano approximation]\label{def:coupled-tamed-euler-peano}
      The tamed Euler--Peano approximation is a pair $(X_{h},K_{h})$ satisfying
      \begin{equation}\label{eq:coupled-tamed-euler-peano}
            X_{h}(t)
            =
            x_{0}
            +
            \int_{0}^{t} b_{h}\big(\overline{X}_{h}(s)\big) \,ds
            +
            \int_{0}^{t} \sigma_{h}\big(\overline{X}_{h}(s)\big) \,dW(s)
            +
            K_{h}(t), \quad t \in [0,T].
      \end{equation}
      Here, $\{X_{h}(t)\}_{t \in [0,T]}$ is an $\overline{D}$-valued continuous adapted process, and $\{K_{h}(t)\}_{t \in [0,T]}$ is a continuous adapted process of bounded variation satisfying $K_{h}(0)=0$. More precisely, there exists a progressively measurable process $\mathbf{n}_{h} \colon  [0,T] \times \Omega \to \R^{d}$ such that
      \begin{equation}\label{eq:numerical-reflection-decomposition}
            K_{h}(t)
            =
            \int_{0}^{t} \mathbf{n}_{h}(s)\,d|K_{h}|(s),
            \quad \mathbf{n}_{h}(s) \in N_{X_{h}(s)}
            \quad d|K_{h}|\text{-a.e.},
      \end{equation}
      and
      \begin{equation}\label{eq:numerical-reflection-support}
            |K_{h}|_{[0,t]}
            =
            \int_{0}^{t} \mathbf{1}_{\{X_{h}(s)\in\partial D\}} \,d|K_{h}|(s),
            \quad t \in [0,T].
      \end{equation}
\end{definition}

The approximation may equivalently be constructed recursively by solving a Skorokhod problem on each temporal subinterval. Suppose that $X_{h}$ and $K_{h}$ have already been constructed on $[0,t_{k}]$. For $t \in [t_{k},t_{k+1}]$, define the continuous driving path
\begin{align}\label{eq:one-step-driving-path}
      Y_{h}^{k}(t)
      :=
      X_{h}(t_{k})
      +
      b_{h}\big(X_{h}(t_{k})\big)(t-t_{k})
      +
      \sigma_{h}\big(X_{h}(t_{k})\big)\big(W(t)-W(t_{k})\big).
\end{align}
Then $\big(X_{h}(t),K_{h}(t)-K_{h}(t_{k})\big), t \in [t_{k},t_{k+1}]$ is the solution of the Skorokhod problem associated with
$Y_{h}^{k}$, with initial value $X_{h}(t_{k})$. Since the coefficients are frozen on each subinterval, the approximation can be constructed recursively by solving a deterministic Skorokhod problem along each realized driving path. This gives the following
well-posedness result.

\begin{proposition}\label{prop:numerical-well-posedness}
      Suppose that Assumptions~\ref{ass:domain} and \ref{ass:coefficients} hold and that $b$ and $\sigma$ are continuous. Then for every $N \in \mathbb{N}$, the tamed Euler--Peano approximation \eqref{eq:coupled-tamed-euler-peano}--\eqref{eq:numerical-reflection-support} admits a unique continuous adapted solution $(X_{h},K_{h})$ on $[0,T]$.
\end{proposition}

\begin{proof}
      We construct the approximation recursively over the temporal grid. Set $X_{h}(0) = x_{0}$ and $K_{h}(0) = 0$. Suppose that $X_{h}$ and $K_{h}$ have already been uniquely constructed on $[0,t_{k}]$ for some $k \in \{0,1,\cdots,N-1\}$. Since $X_{h}(t_{k})$ is $\mathcal{F}_{t_{k}}$-measurable, the random variables $b_{h}\big(X_{h}(t_{k})\big)$ and $\sigma_{h}\big(X_{h}(t_{k})\big)$ are $\mathcal{F}_{t_{k}}$-measurable. Hence, the path $Y_{h}^{k}$ defined by \eqref{eq:one-step-driving-path} is continuous and adapted on $[t_{k},t_{k+1}]$, and satisfies $Y_{h}^{k}(t_{k}) =          X_{h}(t_{k}) \in \overline{D}$. By the pathwise solvability of the Skorokhod problem on convex domains $D$ recalled in Subsection~\ref{subsect:convexdomains}, there exists a unique pair $\{\big(X_{h}(t), K_{h}(t)-K_{h}(t_{k})\big)\}_{t \in [t_{k},t_{k+1}]}$ satisfying
      \begin{equation*}
            X_{h}(t)
            =
            Y_{h}^{k}(t)
            +
            K_{h}(t)-K_{h}(t_{k}),
            \quad
            t \in [t_{k},t_{k+1}],
      \end{equation*}
      together with the corresponding normal-reflection and boundary-support conditions. Repeating this construction for $k = 0, 1, \cdots, N-1$ yields a unique continuous adapted pair $(X_{h},K_{h})$ on $[0,T]$. By construction, this pair satisfies
      \eqref{eq:coupled-tamed-euler-peano}--\eqref{eq:numerical-reflection-support}.
\end{proof}

\subsection{Algebraic properties of the tamed coefficients}\label{subsec:tamed-coefficient-properties}
We record several elementary algebraic properties of the coupled tamed coefficients. These properties will be used in the increment and moment estimates below. The first lemma gives the basic boundedness and growth properties of the tamed coefficients.

\begin{lemma}\label{lem:basic-coupled-taming-properties}
      Suppose that Assumption~\ref{ass:coefficients} holds. Then, for every $h \in (0,1]$ and $x \in \R^{d}$,
      \begin{equation}\label{eq:taming-factor-bounds}
            0 < \lambda_{h}(x) \leq 1.
      \end{equation}
      Moreover, it holds that
      \begin{equation}\label{eq:tamed-coefficients-dominated}
            |b_{h}(x)| \leq |b(x)|,
            \quad
            \|\sigma_{h}(x)\|_{\mathrm{HS}}
            \leq
            \|\sigma(x)\|_{\mathrm{HS}},
      \end{equation}
      and
      \begin{equation}\label{eq:tamed-coefficients-step-bounds}
            h^{1/2}|b_{h}(x)| \leq 1,
            \quad
            h^{1/2} \|\sigma_{h}(x)\|_{\mathrm{HS}}^{2} \leq 1.
      \end{equation}
      In addition, there exists a constant $C > 0$, independent of $h$, such that
      \begin{equation}\label{eq:tamed-coefficients-growth}
            |b_{h}(x)|
            \leq
            C\big(1+|x|^{\gamma_{b}}\big),
            \quad
            \|\sigma_{h}(x)\|_{\mathrm{HS}}
            \leq
            C\big(1+|x|^{\gamma_{\sigma}}\big).
      \end{equation}
\end{lemma}

\begin{proof}
      By the definition of $\lambda_{h}$, we have $0 < \lambda_{h}(x) = \frac{1}{1+h^{1/2}G(x)} \leq 1$, which proves \eqref{eq:taming-factor-bounds}. Hence, $|b_{h}(x)| = \lambda_{h}(x)|b(x)| \leq |b(x)|$, and
      $\|\sigma_{h}(x)\|_{\mathrm{HS}} = \lambda_{h}^{1/2}(x)\|\sigma(x)\|_{\mathrm{HS}} \leq \|\sigma(x)\|_{\mathrm{HS}}$.
      This proves \eqref{eq:tamed-coefficients-dominated}. Since $|b(x)| \leq G(x)$ and $\|\sigma(x)\|_{\mathrm{HS}}^{2} \leq G(x)$, we have
      \begin{gather*}
            h^{1/2}|b_{h}(x)|
            =
            \frac{h^{1/2}|b(x)|}{1 + h^{1/2}G(x)}
            \leq
            \frac{h^{1/2}G(x)}{1 + h^{1/2}G(x)}
            \leq
            1,
            \\
            h^{1/2}\|\sigma_{h}(x)\|_{\mathrm{HS}}^{2}
            =
            h^{1/2}\lambda_{h}(x)\|\sigma(x)\|_{\mathrm{HS}}^{2}
            =
            \frac{h^{1/2}\|\sigma(x)\|_{\mathrm{HS}}^{2}}{1+h^{1/2}G(x)}
            \leq
            \frac{h^{1/2}G(x)}{1+h^{1/2}G(x)}
            \leq
            1.
      \end{gather*}
      Thus, \eqref{eq:tamed-coefficients-step-bounds} holds. Finally, \eqref{eq:tamed-coefficients-growth} follows immediately from \eqref{eq:coefficient-growth}, \eqref{eq:tamed-coefficients-dominated} and the definitions $\gamma_{b} = q_{b} + 1$ and $\gamma_{\sigma} = q_{\sigma} + 1$. 
\end{proof}

The next lemma is the key algebraic consequence of the coupled taming: the high-order coercivity estimate of the original coefficients is inherited by the tamed coefficients.
\begin{lemma}%[Preservation of high-order coercivity]
\label{lem:tamed-high-order-coercivity}
      Suppose that Assumption~\ref{ass:coefficients} holds and let $1 \leq p < p_{\ast}$. Then there exist constants $\varepsilon_{p} > 0$ and $C_{p} > 0$, independent of $h \in (0,1]$, such that
      \begin{align*}
            2\big\langle x-x_{\ast},b_{h}(x) \big\rangle
            +
            (2p-1)\|\sigma_{h}(x)\|_{\mathrm{HS}}^{2}
            \leq
            C_{p}V(x) - \varepsilon_{p}\|\sigma_{h}(x)\|_{\mathrm{HS}}^{2}.
      \end{align*}
\end{lemma}

\begin{proof}
      By Lemma~\ref{lem:high-order-coercivity}, we have
      \begin{align*}
            2 \big\langle x-x_{\ast},b(x) \big\rangle
            +
            \big(2p-1+\varepsilon_{p}\big)
            \|\sigma(x)\|_{\mathrm{HS}}^{2}
            \leq
            C_{p}V(x).
      \end{align*}
      From $b_{h}(x) = \lambda_{h}(x)b(x)$,  $\|\sigma_{h}(x)\|_{\mathrm{HS}}^{2} = \lambda_{h}(x) \|\sigma(x)\|_{\mathrm{HS}}^{2}$ and \eqref{eq:taming-factor-bounds}, it follows that
      \begin{align*}
            &~2\big\langle x-x_{\ast}, b_{h}(x) \big\rangle
            +
            \big(2p-1+\varepsilon_{p}\big)
            \|\sigma_{h}(x)\|_{\mathrm{HS}}^{2}
            \\=&~
            \lambda_{h}(x)
            \big(2\big\langle x-x_{\ast}, b(x) \big\rangle
            +
            \big(2p-1+\varepsilon_{p}\big)
            \|\sigma(x)\|_{\mathrm{HS}}^{2}\big)
            \\\leq&~
            C_{p}\lambda_{h}(x)V(x)
            \\\leq&~
            C_{p}V(x).
      \end{align*}
      Subtracting $\varepsilon_{p}\|\sigma_{h}(x)\|_{\mathrm{HS}}^{2}$ from both sides gives the desired result.
\end{proof}

We also need to quantify the difference between the original coefficients and their tamed versions.

\begin{lemma}\label{lem:coupled-taming-defects}
      Suppose that Assumption~\ref{ass:coefficients} holds. Then for every $h \in (0,1]$ and $x \in \R^{d}$,
      \begin{gather}\label{eq:drift-taming-defect-basic}
            |b(x)-b_{h}(x)| \leq h^{1/2}G(x)|b(x)|,
            \\\label{eq:diffusion-taming-defect-basic}
            \|\sigma(x)-\sigma_{h}(x)\|_{\mathrm{HS}}
            \leq
            h^{1/2}G(x)\|\sigma(x)\|_{\mathrm{HS}}.
      \end{gather}
      In particular, there exists a constant $C>0$, independent of $h$, such that
      \begin{gather}\label{eq:drift-taming-defect-growth}
            |b(x)-b_{h}(x)|
            \leq
            Ch^{1/2}\big(1+|x|^{\widehat{\gamma}+\gamma_{b}}\big),
            \\\label{eq:diffusion-taming-defect-growth}
            \|\sigma(x)-\sigma_{h}(x)\|_{\mathrm{HS}}
            \leq
            Ch^{1/2}\big(1+|x|^{\widehat{\gamma}+\gamma_{\sigma}}\big).
      \end{gather}
\end{lemma}

\begin{proof}
      Since $1 - \lambda_{h}(x) = \frac{h^{1/2}G(x)}{1 + h^{1/2}G(x)} \leq h^{1/2}G(x)$, we obtain
      \begin{align*}
            |b(x)-b_{h}(x)|
            =
            \big(1-\lambda_{h}(x)\big)|b(x)|
%            =
%            \frac{h^{1/2}G(x)}{1+h^{1/2}G(x)}|b(x)|
            \leq
            h^{1/2}G(x)|b(x)|,
      \end{align*}
      which proves \eqref{eq:drift-taming-defect-basic}. Moreover, the inequality
      \begin{equation*}
            1-\lambda_{h}^{1/2}(x)
            =
            \frac{1-\lambda_{h}(x)}{1+\lambda_{h}^{1/2}(x)}
            \leq
            1-\lambda_{h}(x)
            \leq
            h^{1/2}G(x)
      \end{equation*}
      enables us to get
      \begin{align*}
            \|\sigma(x)-\sigma_{h}(x)\|_{\mathrm{HS}}
            =
            \big(1 - \lambda_{h}^{1/2}(x)\big)
            \|\sigma(x)\|_{\mathrm{HS}}
            \leq
            h^{1/2} G(x) \|\sigma(x)\|_{\mathrm{HS}},
      \end{align*}
      proving \eqref{eq:diffusion-taming-defect-basic}. By \eqref{eq:coefficient-growth}, we have
      \begin{align*}
            G(x)
            =
            |b(x)| + \|\sigma(x)\|_{\mathrm{HS}}^{2}
            \leq
            C\big(1 + |x|^{\gamma_{b}} + |x|^{2\gamma_{\sigma}}\big)
            \leq
            C\big(1 + |x|^{\widehat{\gamma}}\big).
      \end{align*}
      It follows that
      \begin{gather*}
            G(x)|b(x)|
            \leq
            C\big(1 + |x|^{\widehat{\gamma}}\big)\big(1 + |x|^{\gamma_{b}}\big)
            \leq
            C\big(1 + |x|^{\widehat{\gamma}+\gamma_{b}}\big),
            \\
            G(x)\|\sigma(x)\|_{\mathrm{HS}}
            \leq
            C\big(1 + |x|^{\widehat{\gamma}}\big)
            \big(1 + |x|^{\gamma_{\sigma}}\big)
            \leq
            C\big(1 + |x|^{\widehat{\gamma}+\gamma_{\sigma}}\big).
      \end{gather*}
      Combining these estimates with \eqref{eq:drift-taming-defect-basic} and \eqref{eq:diffusion-taming-defect-basic} proves
      \eqref{eq:drift-taming-defect-growth} and \eqref{eq:diffusion-taming-defect-growth}.
\end{proof}

\subsection{One-step increment estimates}\label{subsec:coarse-increment-estimates}
Before proving global moment estimates, we establish local increment bounds for the numerical path. The first estimate is conditional and does not require any a priori moment bound for the numerical solution. Writing $Z_{h,k}(t) := X_{h}(t)-X_{h}(t_{k}), t \in [t_{k},t_{k+1}]$, we get
\begin{align}\label{eq:one-step-state-increment-equation}
      Z_{h,k}(t)
      =
      b_{h}(X_{h}(t_{k}))(t-t_{k})
      +
      \sigma_{h}(X_{h}(t_{k}))\big(W(t)-W(t_{k})\big)
      +
      K_{h}(t)-K_{h}(t_{k}).
\end{align}
In what follows, for $k = 0, 1, \cdots, N-1$, we write
\begin{equation*}
      \E_{k}\big[\,\cdot\,\big]
      :=
      \E\left[\,\cdot\,\middle|\mathcal{F}_{t_{k}}\right].
\end{equation*}

The following conditional estimate controls one-step increments in terms of the frozen tamed coefficients.
\begin{lemma}%[Conditional one-step increment estimate]
\label{lem:conditional-one-step-increment}
       Fix $r \geq 2$ and let $\tau$ be a stopping time satisfying $t_{k} \leq \tau \leq t_{k+1}$. Then there exists a constant $C_{r} > 0$, independent of $h$, $k$ and $\tau$, such that
      \begin{align}\label{eq:conditional-one-step-increment}
            \E_{k}\left[ \sup_{t_{k} \leq t \leq t_{k+1}}
            \left|X_{h}(t\wedge\tau)-X_{h}(t_{k})\right|^{r}\right]
            \leq
            C_{r}\big(h^{r}|b_{h}(X_{h}(t_{k}))|^{r}
            +
            h^{r/2}\|\sigma_{h}(X_{h}(t_{k}))\|_{\mathrm{HS}}^{r}\big).
      \end{align}
      Moreover, there exists a constant $C_{r} > 0$, independent of $h$ and $k$, such that
      \begin{equation}\label{eq:coarse-one-step-increment-unconditional}
            \max_{0\leq k\leq N-1}\E
            \left[\sup_{t_{k}\leq t\leq t_{k+1}}
            |X_{h}(t)-X_{h}(t_{k})|^{r}\right]
            \leq
            C_{r}h^{r/4}.
      \end{equation}
\end{lemma}

\begin{proof}
      For $t \in [t_{k},t_{k+1}]$, define $Z^{\tau}_{h,k}(t) := X_{h}(t \wedge \tau) - X_{h}(t_{k})$. Applying It\^{o}'s formula gives
      \begin{align}\label{eq:ito-one-step-increment}
            |Z^{\tau}_{h,k}(t)|^{r}
            =&~
            r \int_{t_{k}}^{t \wedge \tau} |Z_{h,k}(s)|^{r-2}
            \big\langle Z_{h,k}(s),b_{h}(X_{h}(t_{k})) \big\rangle \,ds \notag
            \\&~+
            \frac{r}{2} \int_{t_{k}}^{t \wedge \tau} |Z_{h,k}(s)|^{r-2} 
            \|\sigma_{h}(X_{h}(t_{k}))\|_{\mathrm{HS}}^{2} \,ds \nonumber
            \\&~+
            \frac{r(r-2)}{2}\int_{t_{k}}^{t \wedge \tau}|Z_{h,k}(s)|^{r-4}
            \left|\sigma_{h}(X_{h}(t_{k}))^{\top}Z_{h,k}(s)\right|^{2} \,ds \nonumber
            \\&~+
            r\int_{t_{k}}^{t \wedge \tau} |Z_{h,k}(s)|^{r-2}
            \big\langle Z_{h,k}(s),\sigma_{h}(X_{h}(t_{k}))
            \,dW(s) \big\rangle \nonumber
            \\&~+
            r\int_{t_{k}}^{t \wedge \tau}|Z_{h,k}(s)|^{r-2}
            \big\langle Z_{h,k}(s),dK_{h}(s) \big\rangle.
      \end{align}
      By Lemma~\ref{lem:reflection-monotonicity} and $X_{h}(t_{k}) \in \overline{D}$, we have 
      \begin{align*}
            r\int_{t_{k}}^{t\wedge\tau}|Z_{h,k}(s)|^{r-2}
            \big\langle Z_{h,k}(s),dK_{h}(s) \big\rangle
            \leq
            0.
      \end{align*}
      Together with $\left|\sigma_{h}(X_{h}(t_{k}))^{\top}Z_{h,k}(s)\right|^{2} \leq \|\sigma_{h}(X_{h}(t_{k}))\|_{\mathrm{HS}}^{2} |Z_{h,k}(s)|^{2}$, one gets
      \begin{align*}
            |Z^{\tau}_{h,k}(t)|^{r}
            \leq{}&~
            r\int_{t_{k}}^{t\wedge\tau}|Z_{h,k}(s)|^{r-1}|b_{h}(X_{h}(t_{k}))|\,ds
            +
            \frac{r(r-1)}{2}
            \|\sigma_{h}(X_{h}(t_{k}))\|_{\mathrm{HS}}^{2}
            \int_{t_{k}}^{t\wedge\tau}|Z_{h,k}(s)|^{r-2} \,ds \nonumber
            \\&~
            +
            r\int_{t_{k}}^{t \wedge \tau}|Z_{h,k}(s)|^{r-2}
            \big\langle Z_{h,k}(s),\sigma_{h}(X_{h}(t_{k}))\,dW(s) \big\rangle.
      \end{align*}
      Setting $\mathcal{Z}_{h,k}^{\tau} := \sup\limits_{t_{k} \leq t \leq t_{k+1}}|Z^{\tau}_{h,k}(t)|$ and taking the supremum over $t \in [t_{k}, t_{k+1}]$, we obtain
      \begin{align}\label{eq:one-step-increment-supremum}
            \big(\mathcal{Z}_{h,k}^{\tau}\big)^{r}
            \leq&~
            rh\big(\mathcal{Z}_{h,k}^{\tau}\big)^{r-1}|b_{h}(X_{h}(t_{k}))|
            +
            \frac{r(r-1)}{2}h\big(\mathcal{Z}_{h,k}^{\tau}\big)^{r-2}
            \|\sigma_{h}(X_{h}(t_{k}))\|_{\mathrm{HS}}^{2}  \nonumber
            \\&~+
            r \sup_{t_{k} \leq t \leq t_{k+1}}
            \left|\int_{t_{k}}^{t \wedge \tau}|Z_{h,k}(s)|^{r-2}
            \big\langle Z_{h,k}(s),\sigma_{h}(X_{h}(t_{k}))\,dW(s) \big\rangle\right|.
      \end{align}
      Utilizing Young's inequality shows that for every $\delta > 0$, 
      \begin{gather}\label{eq:one-step-drift-young}
            h\big(\mathcal{Z}_{h,k}^{\tau}\big)^{r-1} |b_{h}(X_{h}(t_{k}))|
            \leq
            \delta \big(\mathcal{Z}_{h,k}^{\tau}\big)^{r}
            +
            C_{r,\delta} h^{r}|b_{h}(X_{h}(t_{k}))|^{r},
            \\\label{eq:one-step-diffusion-fv-young}
            h \big( \mathcal{Z}_{h,k}^{\tau} \big)^{r-2}
            \|\sigma_{h}(X_{h}(t_{k}))\|_{\mathrm{HS}}^{2}
            \leq
            \delta \big( \mathcal{Z}_{h,k}^{\tau} \big)^{r}
            +
            C_{r,\delta} h^{r/2} \|\sigma_{h}(X_{h}(t_{k}))\|_{\mathrm{HS}}^{r}.
      \end{gather}
      Besides, the conditional Burkholder--Davis--Gundy inequality implies
      \begin{align}\label{eq:one-step-martingale-estimate}
            &~\E_{k}\left[ \sup_{t_{k} \leq t \leq t_{k+1}}
            \left|\int_{t_{k}}^{t \wedge \tau} |Z_{h,k}(s)|^{r-2}
            \big\langle Z_{h,k}(s),\sigma_{h}(X_{h}(t_{k}))\,dW(s) 
            \big\rangle \right| \right] \nonumber
            \\\leq&~
            C_{r}\E_{k}\left[ \left(\int_{t_{k}}^{\tau}
            |Z_{h,k}(s)|^{2r-2} \|\sigma_{h}(X_{h}(t_{k}))\|_{\mathrm{HS}}^{2} 
            \,ds \right)^{1/2} \right] \nonumber
            \\\leq&~
            C_{r}\E_{k}\left[ \big(\mathcal{Z}_{h,k}^{\tau}\big)^{r-1}
            h^{1/2} \|\sigma_{h}(X_{h}(t_{k}))\|_{\mathrm{HS}} \right] \nonumber
            \\\leq&~
            \delta \E_{k}\left[ \big(\mathcal{Z}_{h,k}^{\tau}\big)^{r} \right]
            +
            C_{r,\delta} h^{r/2} \|\sigma_{h}(X_{h}(t_{k}))\|_{\mathrm{HS}}^{r},
      \end{align}
       where we have used the fact that $\sigma_{h}(X_{h}(t_{k}))$ is $\mathcal{F}_{t_{k}}$-measurable. Taking conditional expectations in
      \eqref{eq:one-step-increment-supremum}, and then using \eqref{eq:one-step-drift-young}, \eqref{eq:one-step-diffusion-fv-young} and \eqref{eq:one-step-martingale-estimate}, we derive that
      \begin{align*}
            \E_{k} \left[ \big( \mathcal{Z}_{h,k}^{\tau} \big)^{r} \right]
            \leq
            \frac{r(r+3)}{2}\delta \E_{k} 
            \left[ \big(\mathcal{Z}_{h,k}^{\tau}\big)^{r} \right]
            +
            C_{r,\delta}\left(h^{r}|b_{h}(X_{h}(t_{k}))|^{r}
            +
            h^{r/2}\|\sigma_{h}(X_{h}(t_{k}))\|_{\mathrm{HS}}^{r}\right).
      \end{align*}
      Since $r \geq 2$ is fixed and $\delta > 0$ is arbitrary, one can choose $\delta = \delta_{r} > 0$ sufficiently small (e.g., $\delta_{r} = \frac{1}{r(r+3)+1})$ such that $r(r+3)\delta < 1$. In this way, the first term on the right-hand side can be absorbed into the left-hand side, thus yielding \eqref{eq:conditional-one-step-increment}.

      Applying \eqref{eq:conditional-one-step-increment} with
      $\tau=t_{k+1}$ gives
      \begin{align*}
            \E_{k}\left[\sup_{t_{k}\leq t\leq t_{k+1}}
            |X_{h}(t)-X_{h}(t_{k})|^{r} \right]
            \leq
            C_{r}\left(h^{r}|b_{h}(X_{h}(t_{k}))|^{r}
            +
            h^{r/2}\|\sigma_{h}(X_{h}(t_{k}))\|_{\mathrm{HS}}^{r}\right).
      \end{align*}
      Owing to \eqref{eq:tamed-coefficients-step-bounds}, we have $h^{r}|b_{h}(X_{h}(t_{k}))|^{r} \leq h^{r/2}$ and $h^{r/2}\|\sigma_{h}(X_{h}(t_{k}))\|_{\mathrm{HS}}^{r} \leq h^{r/4}$. By $h \in (0,1]$, $h^{r/2} \leq h^{r/4}$ and hence \begin{equation*}
            \E_{k}\left[\sup_{t_{k} \leq t \leq t_{k+1}}
            |X_{h}(t)-X_{h}(t_{k})|^{r}\right]
            \leq
            C_{r}h^{r/4}.
      \end{equation*}
      Taking expectations yields \eqref{eq:coarse-one-step-increment-unconditional}.
\end{proof}

\subsection{Uniform moment estimates}\label{subsec:stopped-grid-moment-estimates}

We now establish uniform moment estimates for the numerical solution. We first derive high-order moment bounds at the temporal grid points by means of a one-step Lyapunov estimate and a discrete stopping argument. These bounds are then used to sharpen the local increment estimate and to extend the moment estimate to the whole time interval. For simplicity, we write
\begin{align*}
      V_{h,k} := V(X_{h}(t_{k})) = 1 + |X_{h}(t_{k})-x_{\ast}|^{2}.
\end{align*}

\begin{lemma}\label{lem:one-step-lyapunov}
      Suppose that Assumptions~\ref{ass:domain} and \ref{ass:coefficients} hold and let $2 \leq P < p_{\ast}$. Then there exist constants $C_{P} > 0$ and $c_{P} > 0$, independent of $h \in (0,1]$ and $k$, such that
      \begin{align*}
            \E_{k}\left[V_{h,k+1}^{P}\right]
            +
            c_{P}\|\sigma_{h}(X_{h}(t_{k}))\|_{\mathrm{HS}}^{2}
            \E_{k}\left[\int_{t_{k}}^{t_{k+1}}
            V^{P-1}\big(X_{h}(s)\big)\,ds\right]
            \leq
            \big(1 + C_{P}h\big)V_{h,k}^{P}.
      \end{align*}
\end{lemma}

\begin{proof}
      Recalling $Z_{h,k}(s) = X_{h}(s) - X_{h}(t_{k}), s \in [t_{k},t_{k+1}]$ and setting $\widehat{Z}_{h}(s) := X_{h}(s) - x_{\ast}$ and $z_{h,k} := X_{h}(t_{k}) - x_{\ast}$ yield $\widehat{Z}_{h}(s) = z_{h,k} + Z_{h,k}(s)$. Applying It\^{o}'s formula to $V^{P}\big(X_{h}(t)\big)$ on $[t_{k},t]$ with $t \in [t_{k},t_{k+1}]$ gives
      \begin{align}\label{eq:ito-numerical-one-step-lyapunov}
            V^{P}\big(X_{h}(t)\big) \nonumber
            =&~
            V_{h,k}^{P}
            +
            P\int_{t_{k}}^{t}V^{P-1}\big(X_{h}(s)\big)
            \big(2 \big\langle \widehat{Z}_{h}(s),b_{h}(X_{h}(t_{k})) \big\rangle
            +
            \|\sigma_{h}(X_{h}(t_{k}))\|_{\mathrm{HS}}^{2} \big)\,ds \nonumber
            \\&~
            +
            2P(P-1)\int_{t_{k}}^{t}V^{P-2}\big(X_{h}(s)\big)
            \big|\sigma_{h}(X_{h}(t_{k}))^{\top}\widehat{Z}_{h}(s)\big|^{2}\,ds \nonumber
            \\&~
            +
            2P\int_{t_{k}}^{t}V^{P-1}\big(X_{h}(s)\big)
            \big\langle \widehat{Z}_{h}(s),
            \sigma_{h}(X_{h}(t_{k}))\,dW(s) \big\rangle \nonumber
            \\&~
            +
            2P\int_{t_{k}}^{t} V^{P-1}\big(X_{h}(s)\big)
            \big\langle \widehat{Z}_{h}(s),dK_{h}(s) \big\rangle.
      \end{align}
      Owing to \eqref{eq:reflection-one-path} with $y = x_{\ast}$ and $$ \big|\sigma_{h}(X_{h}(t_{k}))^{\top} \widehat{Z}_{h}(s)\big|^{2}
      \leq
      \|\sigma_{h}(X_{h}(t_{k}))\|_{\mathrm{HS}}^{2}|\widehat{Z}_{h}(s)|^{2}
      \leq
      \|\sigma_{h}(X_{h}(t_{k}))\|_{\mathrm{HS}}^{2}V\big(X_{h}(s)\big),$$ we take  conditional expectations with $t = t_{k+1}$ in \eqref{eq:ito-numerical-one-step-lyapunov} to obtain
      \begin{align*}
            \E_{k}\left[V_{h,k+1}^{P}\right] \notag
            \leq
            V_{h,k}^{P}
            +
            P\E_{k}\int_{t_{k}}^{t_{k+1}}V^{P-1}\big(X_{h}(s)\big)
            \big(2\big\langle \widehat{Z}_{h}(s),
            b_{h}(X_{h}(t_{k})) \big\rangle
            +
            (2P-1)\|\sigma_{h}(X_{h}(t_{k}))\|_{\mathrm{HS}}^{2}\big)\,ds.
      \end{align*}
      By $\widehat{Z}_{h}(s) = z_{h,k} + Z_{h,k}(s)$, we have
      $2\big\langle \widehat{Z}_{h}(s),b_{h}(X_{h}(t_{k})) \big\rangle
            =
            2 \big\langle z_{h,k},b_{h}(X_{h}(t_{k})) \big\rangle
            +
            2 \big\langle Z_{h,k}(s),b_{h}(X_{h}(t_{k})) \big\rangle $.
      Besides, Lemma~\ref{lem:tamed-high-order-coercivity} gives
      \begin{align*}
            2\big\langle z_{h,k},b_{h}(X_{h}(t_{k}))\big\rangle
            +
            (2P-1)\|\sigma_{h}(X_{h}(t_{k}))\|_{\mathrm{HS}}^{2}
            \leq
            C_{P}V_{h,k}
            -
            \varepsilon_{P}\|\sigma_{h}(X_{h}(t_{k}))\|_{\mathrm{HS}}^{2}.
      \end{align*}
      It follows that
      \begin{align}\label{eq:one-step-lyapunov-IJD}
            \E_{k}\left[V_{h,k+1}^{P}\right]
            \leq
            V_{h,k}^{P} + C_{P}I_{h,k}
            -
            P\varepsilon_{P}D_{h,k} + 2PJ_{h,k}
      \end{align}
      with
      \begin{gather*}
            I_{h,k}
            :=
            \E_{k}\left[\int_{t_{k}}^{t_{k+1}}
            V_{h,k} V^{P-1}\big(X_{h}(s)\big) \,ds\right],
            \\
            D_{h,k}
            :=
            \|\sigma_{h}(X_{h}(t_{k}))\|_{\mathrm{HS}}^{2}\,
            \E_{k} \left[\int_{t_{k}}^{t_{k+1}}
            V^{P-1}\big(X_{h}(s)\big) \,ds\right],
            \\
            J_{h,k}
            :=
            |b_{h}(X_{h}(t_{k}))|\,\E_{k}
            \left[\int_{t_{k}}^{t_{k+1}}
            V^{P-1}\big(X_{h}(s)\big) |Z_{h,k}(s)|\,ds\right].
      \end{gather*}

      We estimate the three terms in \eqref{eq:one-step-lyapunov-IJD}. Using $V\big(X_{h}(s)\big) \leq 2\big(V_{h,k} + |Z_{h,k}(s)|^{2}\big)$ and Young's inequality yields
      \begin{equation*}
            V_{h,k}V^{P-1}\big(X_{h}(s)\big)
            \leq
            C_{P}\big(V_{h,k}^{P} + |Z_{h,k}(s)|^{2P}\big).
      \end{equation*}
      Together with Lemma~\ref{lem:conditional-one-step-increment}, $h|b_{h}(X_{h}(t_{k}))| \leq h^{1/2}$, 
      $h^{1/2}\|\sigma_{h}(X_{h}(t_{k}))\|_{\mathrm{HS}} \leq h^{1/4}$
      and $V_{h,k}^{P} \geq 1$, we deduce that
      \begin{align}\label{eq:Ihk-final-bound}
            I_{h,k}
            \leq&~
            C_{P}hV_{h,k}^{P}
            +
            C_{P}h\E_{k}
            \left[\sup_{t_{k}\leq s\leq t_{k+1}}
            |Z_{h,k}(s)|^{2P}\right] \notag
            \\\leq&~
            C_{P}h V_{h,k}^{P}
            +
            C_{P}h\Big((h|b_{h}(X_{h}(t_{k}))|)^{2P}
            +
            \big(h^{1/2}\|\sigma_{h}(X_{h}(t_{k}))
            \|_{\mathrm{HS}}\big)^{2P}\Big) \notag
%            \\\leq&~
%            C_{P}V^{P}(X_{h}(t_{k})) + C_{P}h \notag
            \\\leq&~
            C_{P}h V_{h,k}^{P}.
      \end{align}

      Concerning $J_{h,k}$, the inequality $V^{P-1}\big(X_{h}(s)\big) \leq C_{P} \big(V_{h,k}^{P-1} + |Z_{h,k}(s)|^{2P-2}\big)$ implies
      \begin{equation}\label{eq:Jhk-splitting}
            J_{h,k}
            \leq
            C_{P}\big(J_{h,k}^{(1)}+J_{h,k}^{(2)}\big)
      \end{equation}
      with
      \begin{gather*}
            J_{h,k}^{(1)}
            :=
            hV_{h,k}^{P-1}|b_{h}(X_{h}(t_{k}))|
            \E_{k}\left[\sup_{t_{k}\leq s\leq t_{k+1}}|Z_{h,k}(s)|\right],
            \\
            J_{h,k}^{(2)}
            :=
            h|b_{h}(X_{h}(t_{k}))|\E_{k}
            \left[ \sup_{t_{k}\leq s\leq t_{k+1}}|Z_{h,k}(s)|^{2P-1}\right].
      \end{gather*}
      Making use of Lemma~\ref{lem:conditional-one-step-increment}, Young's inequality, $h|b_{h}(X_{h}(t_{k}))|^{2}\leq1$ and $V_{h,k} \geq 1$ shows that for every $\delta > 0$,
      \begin{align}\label{eq:J1-with-Vk-diffusion}
            J_{h,k}^{(1)}
            \leq&~
            Ch^{2}V_{h,k}^{P-1}|b_{h}(X_{h}(t_{k}))|^{2}
            +
            C h^{3/2} V_{h,k}^{P-1} |b_{h}(X_{h}(t_{k}))|
            \|\sigma_{h}(X_{h}(t_{k}))\|_{\mathrm{HS}} \notag
            \\\leq&~
            Ch^{2}V_{h,k}^{P-1}|b_{h}(X_{h}(t_{k}))|^{2}
            +
            C V_{h,k}^{P-1} \big(C_{\delta}h^{2}|b_{h}(X_{h}(t_{k}))|^{2}
            +
            \delta h \|\sigma_{h}(X_{h}(t_{k}))\|_{\mathrm{HS}}^{2}\big) \notag
            \\\leq&~
            C_{\delta}h V_{h,k}^{P}
            +
            C \delta h V_{h,k}^{P-1}
            \|\sigma_{h}(X_{h}(t_{k}))\|_{\mathrm{HS}}^{2}.
      \end{align}
      We now compare the last term in \eqref{eq:J1-with-Vk-diffusion} with $D_{h,k}$. From $V_{h,k}^{P-1} \leq C_{P}\big(V^{P-1}\big(X_{h}(s)\big) +  |Z_{h,k}(s)|^{2P-2}\big)$, Lemma~\ref{lem:conditional-one-step-increment}, $h^{1/2} \|\sigma_{h}(X_{h}(t_{k}))\|_{\mathrm{HS}}^{2} \leq 1$ and $h|b_{h}(X_{h}(t_{k}))| \leq h^{1/2}$, it follows that
      \begin{align}\label{eq:Vk-diffusion-final-comparison}
            &~h V_{h,k}^{P-1} \|\sigma_{h}(X_{h}(t_{k}))\|_{\mathrm{HS}}^{2}
            =
            \|\sigma_{h}(X_{h}(t_{k}))\|_{\mathrm{HS}}^{2}\, 
            \E_{k}\bigg[\int_{t_{k}}^{t_{k+1}} V_{h,k}^{P-1} \,ds\bigg] \notag
            \\\leq&~
            C_{P}D_{h,k}
            +
            C_{P} h \|\sigma_{h}(X_{h}(t_{k}))\|_{\mathrm{HS}}^{2}\,
            \E_{k} \left[\sup_{t_{k} \leq s\leq t_{k+1}} 
            |Z_{h,k}(s)|^{2P-2} \right] \notag
            \\\leq&~
            C_{P}D_{h,k}
            +
            C_{P} h \|\sigma_{h}(X_{h}(t_{k}))\|_{\mathrm{HS}}^{2} 
            \Big((h|b_{h}(X_{h}(t_{k}))|)^{2P-2} \notag
            \\&~+
            \big(h^{1/2}\|\sigma_{h}(X_{h}(t_{k}))
            \|_{\mathrm{HS}}\big)^{2P-2}\Big) \notag
            \\\leq&~
            C_{P}D_{h,k} + C_{P}h.
      \end{align}
      Combining \eqref{eq:J1-with-Vk-diffusion} and \eqref{eq:Vk-diffusion-final-comparison} and using $V_{h,k}^{P} \geq 1$ result in
      \begin{equation}\label{eq:J1-final}
            J_{h,k}^{(1)}
            \leq
            C_{P,\delta} h V_{h,k}^{P}
            +
            C_{P}\delta D_{h,k}.
      \end{equation}
      For $J_{h,k}^{(2)}$, Lemma~\ref{lem:conditional-one-step-increment},
      $h|b_{h}(X_{h}(t_{k}))| \leq h^{1/2}$ and $h^{1/2}\|\sigma_{h}(X_{h}(t_{k}))\|_{\mathrm{HS}} \leq h^{1/4}$ yield
      \begin{align}\label{eq:J2-final}
            J_{h,k}^{(2)}
            \leq
            C_{P}h|b_{h}(X_{h}(t_{k}))|\big((h|b_{h}(X_{h}(t_{k}))|)^{2P-1}
            +
            \big(h^{1/2}\|\sigma_{h}(X_{h}(t_{k}))\|_{\mathrm{HS}}\big)^{2P-1}\big)
            \leq
            C_{P}h.
      \end{align}
      It follows from \eqref{eq:Jhk-splitting}, \eqref{eq:J1-final} and \eqref{eq:J2-final} that
      \begin{equation}\label{eq:Jhk-final-bound}
            J_{h,k}
            \leq
            C_{P,\delta} h V_{h,k}^{P}
            +
            C_{P}\delta D_{h,k}.
      \end{equation}

      Substituting \eqref{eq:Ihk-final-bound} and \eqref{eq:Jhk-final-bound} into \eqref{eq:one-step-lyapunov-IJD}, we obtain
      \begin{align*}
            \E_{k}\left[V_{h,k+1}^{P}\right]
            \leq
            \big(1+C_{P,\delta}h\big)V_{h,k}^{P}
            -
            \big(P\varepsilon_{P}-C_{P}\delta\big)D_{h,k}.
      \end{align*}
      Choosing $\delta > 0$ sufficiently small so that $P\varepsilon_{P} - C_{P}\delta > 0$ proves the desired result.
\end{proof}

To iterate the one-step estimate without assuming global moment bounds in advance, we introduce a discrete stopping index and obtain the following stopped estimate.

\begin{proposition}%[Stopped grid-point moment estimate]
\label{prop:stopped-grid-moment}
      Suppose that the assumptions of Lemma~\ref{lem:one-step-lyapunov} hold. Let $R > V(x_{0})$ and define the discrete stopping index $\nu_{R} := \inf\big\{k \in \{0,1,\cdots,N\}: V_{h,k} \geq R \big\}$ with the convention $\inf\varnothing := N + 1$. Then
      \begin{equation}\label{eq:stopped-grid-P-moment}
            \max_{0\leq k\leq N}\E\left[V_{h,k\wedge\nu_{R}}^{P}\right]
            \leq
            C_{P,T}V^{P}(x_{0}),
      \end{equation}
      where $C_{P,T}>0$ is independent of $h$ and $R$.
\end{proposition}

\begin{proof}
      We first observe that $\nu_{R}$ is a stopping time with respect to the discrete filtration $            \big(\mathcal{F}_{t_{k}}\big)_{k=0}^{N}$. Indeed, for every $k\in\{0,1,\cdots,N\}$, we have
      \begin{equation*}
            \big\{\nu_{R} \leq k\big\}
            =
            \bigcup_{j=0}^{k}\big\{V_{h,j}\geq R\big\}.
      \end{equation*}
      Since $X_{h}(t_{j})$ is $\mathcal{F}_{t_{j}}$-measurable and $\mathcal{F}_{t_{j}} \subset \mathcal{F}_{t_{k}}$ for $0 \leq j \leq k$, it follows that $\big\{\nu_{R} \leq k\big\} \in \mathcal{F}_{t_{k}}$.

      For $k\in\{0,1,\cdots,N-1\}$, define $A_{k} := \big\{k < \nu_{R}\big\} \in \mathcal{F}_{t_{k}}$. On the event $A_{k}$, the stopping index has not yet been reached. Since $\nu_{R}$ is integer-valued, $k<\nu_{R}$ implies $\nu_{R} \geq k+1$, and hence
      \begin{equation*}
            k\wedge\nu_{R} = k,
            \quad
            (k+1)\wedge\nu_{R} = k+1
            \quad
            \text{on }A_{k}.
      \end{equation*}
      On the complementary event $A_{k}^{c}=\{k\geq\nu_{R}\}$, the stopping has already occurred, and therefore
      \begin{equation*}
            k\wedge\nu_{R} = (k+1)\wedge\nu_{R} = \nu_{R}
            \quad
            \text{on }A_{k}^{c}.
      \end{equation*}
      It follows that
      \begin{align}
            V_{h,(k+1)\wedge\nu_{R}}^{P}
            ={}&~
            \mathbf{1}_{A_{k}}
            V_{h,k+1}^{P}
            +
            \mathbf{1}_{A_{k}^{c}}
            V_{h,k\wedge\nu_{R}}^{P}.
            \label{eq:stopped-variable-decomposition}
      \end{align}
      Since 
      \begin{equation*}
            V_{h,k\wedge\nu_{R}}^{P}
            =
            \sum_{j=0}^{k-1}
            V_{h,j}^{P}
            \mathbf{1}_{\{\nu_{R}=j\}}
            +
            V_{h,k}^{P}
            \mathbf{1}_{\{\nu_{R}\geq k\}},
      \end{equation*}
      and every term on the right-hand side is $\mathcal{F}_{t_{k}}$-measurable, we see that $V_{h,k\wedge\nu_{R}}^{P}$ is $\mathcal{F}_{t_{k}}$-measurable. Taking conditional expectations in \eqref{eq:stopped-variable-decomposition} and applying
      $A_{k} \in \mathcal{F}_{t_{k}}$ gives
      \begin{align*}
            \E_{k}\left[V_{h,(k+1)\wedge\nu_{R}}^{P}\right]
            =
            \mathbf{1}_{A_{k}}\E_{k}\left[V_{h,k+1}^{P}\right]
            +
            \mathbf{1}_{A_{k}^{c}}V_{h,k\wedge\nu_{R}}^{P}
            \leq
            \mathbf{1}_{A_{k}}\big(1+C_{P}h\big)V_{h,k}^{P}
            +
            \mathbf{1}_{A_{k}^{c}}V_{h,k\wedge\nu_{R}}^{P},
      \end{align*}
      where Lemma \ref{lem:one-step-lyapunov} has been used. Together with $V_{h,k\wedge\nu_{R}}^{P} = V_{h,k}^{P}$ on $A_{k}$, we have
      \begin{align*}
            \E_{k}\left[V_{h,(k+1)\wedge\nu_{R}}^{P}\right]
            \leq&~
            \big(1+C_{P}h\big)
            \big(\mathbf{1}_{A_{k}}V_{h,k}^{P}
            +
            \mathbf{1}_{A_{k}^{c}}V_{h,k\wedge\nu_{R}}^{P}\big)
            \\=&~
            \big(1+C_{P}h\big)
            \big(\mathbf{1}_{A_{k}}V_{h,k\wedge\nu_{R}}^{P}
            +
            \mathbf{1}_{A_{k}^{c}}V_{h,k\wedge\nu_{R}}^{P}\big)
            \\=&~
            \big(1+C_{P}h\big)
            V_{h,k\wedge\nu_{R}}^{P}.
      \end{align*}
      Taking expectations and using the tower property of conditional expectation yields
      \begin{equation}\label{eq:unconditional-stopped-recursion}
            \E\left[V_{h,(k+1)\wedge\nu_{R}}^{P}\right]
            \leq
            \big(1 + C_{P}h\big)
            \E\left[V_{h,k\wedge\nu_{R}}^{P}\right].
      \end{equation}
      Since $R > V(x_{0}) = V_{h,0}$, one has $\nu_{R} \geq 1$, and hence $V_{h,0\wedge\nu_{R}}^{P} = V_{h,0}^{P} = V^{P}(x_{0})$. Iterating \eqref{eq:unconditional-stopped-recursion}, we obtain that for every $k \in \{0,1,\cdots,N\}$,
      \begin{align*}
            \E\left[V_{h,k\wedge\nu_{R}}^{P}\right]
            \leq
            \big(1+C_{P}h\big)^{k}V^{P}(x_{0})
            \leq
            \exp\big(C_{P}kh\big)V^{P}(x_{0})
            \leq
            \exp\big(C_{P}T\big)V^{P}(x_{0}).
      \end{align*}
      The constant $C_{P,T} := \exp\big(C_{P}T\big)$ is independent of both $h$ and $R$, and thus proving \eqref{eq:stopped-grid-P-moment}.
\end{proof}

The stopped estimate yields a maximal moment bound at the grid points through a tail-probability argument.
\begin{proposition}\label{prop:maximal-grid-moment}
      Suppose that Assumptions~\ref{ass:domain} and \ref{ass:coefficients} hold, and let $2 \leq P < p_{\ast}$. Then for every $1 \leq p < P$, there exists a constant $C_{p,P,T}>0$, independent of $h \in (0,1]$, such that
      \begin{equation}\label{eq:maximal-grid-moment}
            \E\left[\max_{0\leq k\leq N}V_{h,k}^{p}\right]
            \leq
            C_{p,P,T}V^{p}(x_{0}).
      \end{equation}
      Consequently,
      \begin{equation}\label{eq:maximal-grid-absolute-moment}
            \E\left[\max_{0\leq k\leq N}|X_{h}(t_{k})|^{2p}\right]
            \leq
            C_{p,P,T}\big(1+|x_{0}|^{2p}\big).
      \end{equation}
\end{proposition}

\begin{proof}
      Define the maximal Lyapunov value over the temporal grid by $M_{h,N} := \max_{0\leq k\leq N}V_{h,k}$. We first derive a tail estimate for $M_{h,N}$. Fix $R > V(x_{0})$ and recall the stopping index
      \begin{equation*}
            \nu_{R}
            :=
            \inf\big\{k \in \{0,1,\cdots,N\}:
            V_{h,k} \geq R \big\}
      \end{equation*}
      with the convention $\inf\varnothing := N+1$. On the event $\big\{M_{h,N} \geq R\big\}$, there exists at least one index $k \in \{0,1,\cdots,N\}$ such that $V_{h,k} \geq R$, which implies $\nu_{R} \leq N$, $V_{h,N\wedge\nu_{R}} = V_{h,\nu_{R}} \geq R$ and thus 
      \begin{equation*}
            R^{P}\mathbf{1}_{\{M_{h,N}\geq R\}}
            \leq
            R^{P}\mathbf{1}_{\{\nu_{R}\leq N\}}
            \leq
            V_{h,N \wedge \nu_{R}}^{P}\mathbf{1}_{\{\nu_{R}\leq N\}}
            \leq
            V_{h,N \wedge \nu_{R}}^{P}.
      \end{equation*}
      Taking expectations and applying Proposition \ref{prop:stopped-grid-moment} with $k=N$ gives
      \begin{align}\label{eq:maximum-tail-large-R}
            \mathbb{P}\big(M_{h,N} \geq R\big)
            \leq
            \frac{\E\left[V_{h,N \wedge \nu_{R}}^{P}\right]}{R^{P}}
            \leq
            \frac{C_{P,T}V^{P}(x_{0})}{R^{P}}.
      \end{align}
      Since the constant in Proposition~\ref{prop:stopped-grid-moment} has been chosen as $C_{P,T} = \exp\big(C_{P}T\big) \geq 1$, one has
      \begin{equation*}
            \frac{C_{P,T}V^{P}(x_{0})}{R^{P}}
            \geq
            1,
            \quad 0<R\leq V(x_{0}),
      \end{equation*} 
      which means that one can use the trivial estimate $\mathbb{P}\big(M_{h,N}\geq R\big) \leq 1$ for $R \in (0,V(x_{0})]$. Hence, we obtain that for every $R>0$,
      \begin{equation}\label{eq:grid-maximum-tail}
            \mathbb{P}\big(M_{h,N}\geq R\big)
            \leq
            \min\left\{1,\frac{C_{P,T}V^{P}(x_{0})}{R^{P}}\right\}.
      \end{equation}

%      We now use the tail-integral representation. 
%      For every nonnegative random variable $Y$ and every $p>0$, one has
%      \begin{equation}\label{eq:tail-integral-representation}
%            \E\left[Y^{p}\right]
%            =
%            p\int_{0}^{\infty}R^{p-1}
%            \mathbb{P}\big(Y \geq R \big)\,dR.
%      \end{equation}
%      Indeed,
%      \begin{equation*}
%            Y^{p}
%            =
%            \int_{0}^{\infty} pR^{p-1}
%            \mathbf{1}_{\{Y\geq R\}} \,dR,
%      \end{equation*}
%      and \eqref{eq:tail-integral-representation} follows from Tonelli's theorem.

      Setting $A := C_{P,T}^{1/P}V(x_{0}) \geq V(x_{0})$ and using the tail-integral representation for nonnegative random variables (see, e.g., \cite[Exercise~2.2.7]{durrett2019probability}), we obtain
      \begin{align*}
            \E\left[M_{h,N}^{p}\right]
            =&~
            p\int_{0}^{\infty}R^{p-1}\mathbb{P}
            \big(M_{h,N} \geq R\big)\,dR
            \\=&~
            p\int_{0}^{A}R^{p-1}\mathbb{P}
            \big(M_{h,N} \geq R\big)\,dR
            +
            p\int_{A}^{\infty}R^{p-1}
            \mathbb{P}\big(M_{h,N} \geq R\big)\,dR
            \\\leq{}&~
            p\int_{0}^{A}R^{p-1}\,dR
            +
            pC_{P,T}V^{P}(x_{0})
            \int_{A}^{\infty}R^{p-P-1}\,dR.
      \end{align*}
      Since $p < P$, the second improper integral is finite and satisfies
      \begin{equation*}
            \int_{A}^{\infty} R^{p-P-1} \,dR
            =
            \frac{A^{p-P}}{P-p}.
      \end{equation*}
      It follows that
      \begin{align*}
            \E\left[M_{h,N}^{p}\right]
            \leq
            A^{p}
            +
            \frac{pA^{p-P}}{P-p}C_{P,T}V^{P}(x_{0})
            \leq
            \left(1+\frac{p}{P-p}\right)
            C_{P,T}^{p/P}V^{p}(x_{0})
            =
            \frac{P}{P-p}C_{P,T}^{p/P}V^{p}(x_{0}),
      \end{align*}
      which yields \eqref{eq:maximal-grid-moment}. Finally, \eqref{eq:maximal-grid-absolute-moment} follows from $|x|^{2p} \leq C_{p}V^{p}(x)$ and $V^{p}(x_{0}) \leq C_{p}(1 + |x_{0}|^{2p})$. 
\end{proof}

With the maximal grid-point moment estimate at hand, the deterministic taming bound for the diffusion coefficient can be replaced by its polynomial growth bound in the local increment analysis. As a result, the preliminary order $h^{1/4}$ in Lemma~\ref{lem:conditional-one-step-increment} is improved to the standard local order $h^{1/2}$, leading to the following sharp one-step increment estimate.

\begin{proposition}%[Sharp one-step increment estimate]
\label{prop:sharp-one-step-increment}
      Suppose that Assumptions~\ref{ass:domain} and \ref{ass:coefficients} hold. Let $r\geq2$ and assume that there exists an exponent $P$ such that
      \begin{equation}\label{eq:sharp-increment-moment-condition}
            \frac{r\gamma}{2} < P < p_{\ast}, \quad P \geq 2.
      \end{equation}
      Then there exists a constant $C_{r,P,T} > 0$, independent of $h \in (0,1]$, such that
      \begin{equation}\label{eq:sharp-one-step-increment}
            \max_{0 \leq k \leq N-1}\E
            \left[\sup_{t_{k} \leq t \leq t_{k+1}}
            |X_{h}(t)-X_{h}(t_{k})|^{r}\right]
            \leq
            C_{r,P,T}h^{r/2}.
      \end{equation}
\end{proposition}

\begin{proof}
      Applying Lemma~\ref{lem:conditional-one-step-increment}, \eqref{eq:tamed-coefficients-growth}, $\gamma_{b} \leq \gamma$ and $\gamma_{\sigma} \leq \gamma$ leads to 
      \begin{align*}
            \E\left[\sup_{t_{k}\leq t\leq t_{k+1}}
            |X_{h}(t)-X_{h}(t_{k})|^{r}\right]
            \leq&~
            C_{r}h^{r}\E\left[|b_{h}(X_{h}(t_{k}))|^{r}\right]
            +
            C_{r}h^{r/2}\E
            \left[\|\sigma_{h}(X_{h}(t_{k}))\|_{\mathrm{HS}}^{r}\right].
            \\\leq&~
            C_{r}h^{r}
            \big(1 + \E\left[|X_{h}(t_{k})|^{r\gamma_{b}}\right]\big)
            +
            C_{r}h^{r/2}\big(1 + \E\left[|X_{h}(t_{k})|^{r\gamma_{\sigma}}\right]\big)
            \\\leq&~
            C_{r}(h^{r}+h^{r/2})
            \big(1 + \E\left[|X_{h}(t_{k})|^{r\gamma}\right]\big).
      \end{align*}
      Together with $|x|^{r\gamma} \leq C_{r,\gamma}V^{r\gamma/2}(x)$ for all $x \in \R^{d}$, the condition \eqref{eq:sharp-increment-moment-condition} allows us to apply Proposition~\ref{prop:maximal-grid-moment} with the moment exponent $r\gamma/2$ to obtain
      \begin{equation*}
            \E\left[\sup_{t_{k}\leq t\leq t_{k+1}}
            |X_{h}(t)-X_{h}(t_{k})|^{r}\right]
            \leq
            C_{r,P,T}\big(h^{r} + h^{r/2}\big).
      \end{equation*}
      Since $h \in (0,1]$ and $r \geq 2$, we have $h^{r} \leq h^{r/2}$. Therefore, \eqref{eq:sharp-one-step-increment} follows.
\end{proof}

Finally, we extend the maximal grid-point estimate to the whole time interval. For this purpose, the coarse increment estimate is sufficient and avoids any additional polynomial moment requirement.
\begin{proposition}%[Continuous-time uniform moment estimate]
\label{prop:continuous-time-uniform-moment}
      Suppose that Assumptions~\ref{ass:domain} and \ref{ass:coefficients} hold. Let $p \geq 1$ and assume that there exists an exponent $P$ such that
      \begin{equation}\label{eq:continuous-time-moment-condition}
            \max\big\{p,2\big\} < P < p_{\ast}.
      \end{equation}
      Then for every $T > 0$, there exists a constant $C_{p,P,T} > 0$, independent of $h \in (0,1]$, such that
      \begin{equation}\label{eq:continuous-time-uniform-moment}
            \sup_{0<h\leq1}\E
            \left[\sup_{0\leq t\leq T}
            V^{p}\big(X_{h}(t)\big)\right]
            \leq
            C_{p,P,T}V^{p}(x_{0}).
      \end{equation}
      Consequently,
      \begin{equation}\label{eq:continuous-time-absolute-moment}
            \sup_{0<h\leq1}\E\left[
            \sup_{0\leq t\leq T}|X_{h}(t)|^{2p}\right]
            \leq
            C_{p,P,T}\big(1+|x_{0}|^{2p}\big).
      \end{equation}
\end{proposition}

\begin{proof}
      For $k = 0,1,\cdots,N-1$, we define $\mathcal{I}_{h,k} := \sup_{t_{k}\leq t\leq t_{k+1}}|X_{h}(t)-X_{h}(t_{k})|$. It follows that for every $t \in [t_{k},t_{k+1}]$,
      \begin{align*}
            V\big(X_{h}(t)\big)
%            =
%            1 + |X_{h}(t)-x_{\ast}|^{2}
            \leq
            1 + 2|X_{h}(t_{k})-x_{\ast}|^{2}
            +
            2|X_{h}(t)-X_{h}(t_{k})|^{2}
            \leq
            2V_{h,k} + 2\mathcal{I}_{h,k}^{2},
      \end{align*}
      which in combination with Proposition~\ref{prop:maximal-grid-moment} and \eqref{eq:continuous-time-moment-condition} implies
      \begin{align}\label{eq:path-maximum-grid-increment}
            \E\left[\sup_{0\leq t\leq T}V^{p}\big(X_{h}(t)\big)\right]
            \leq&~
            C_{p}\E\left[\max_{0\leq k\leq N}V_{h,k}^{p}\right]
            +
            C_{p}\E\left[\max_{0\leq k\leq N-1}\mathcal{I}_{h,k}^{2p}\right] \notag
            \\\leq&~
            C_{p,P,T}V^{p}(x_{0})
            +
            C_{p}\E\left[\max_{0\leq k\leq N-1}\mathcal{I}_{h,k}^{2p}\right].
      \end{align}
      It remains to estimate the maximal excursion between grid points. By H\"{o}lder's inequality, one has
      \begin{equation*}
            \E\left[\max_{0\leq k\leq N-1}\mathcal{I}_{h,k}^{2p}\right]
            \leq
            \left(\E\left[\max_{0\leq k\leq N-1}
            \mathcal{I}_{h,k}^{2P}\right]\right)^{p/P}.
      \end{equation*}
      Since the maximum of nonnegative numbers is bounded by their sum, \eqref{eq:coarse-one-step-increment-unconditional} gives
      \begin{align*}
            \E\left[\max_{0\leq k\leq N-1}\mathcal{I}_{h,k}^{2P}\right]
            \leq
            \sum_{k=0}^{N-1}\E\left[\mathcal{I}_{h,k}^{2P}\right]
            \leq
            C_{P}Nh^{P/2}
            =
            C_{P}Th^{P/2-1}.
      \end{align*}
      From $P > 2$ and $h \in (0,1]$, $h^{P/2-1} \leq 1$, it follows that
      \begin{equation}\label{eq:max-increment-uniform}
            \E\left[\max_{0\leq k\leq N-1}
            \mathcal{I}_{h,k}^{2p}\right]
            \leq
            C_{p,P,T}.
      \end{equation}
      Combining \eqref{eq:path-maximum-grid-increment} and \eqref{eq:max-increment-uniform}, we obtain
      \begin{equation*}
            \E\left[\sup_{0\leq t\leq T}
            V^{p}\big(X_{h}(t)\big)\right]
            \leq
            C_{p,P,T}\big(1+V^{p}(x_{0})\big).
      \end{equation*}
      Observing $V(x_{0}) \geq 1$, the constant term may be absorbed into $V^{p}(x_{0})$, proving \eqref{eq:continuous-time-uniform-moment}. Finally, \eqref{eq:continuous-time-absolute-moment} follows from $|x|^{2p} \leq C_{p}V^{p}(x)$ and $V^{p}(x_{0}) \leq C_{p}\big(1 + |x_{0}|^{2p}\big)$. 
\end{proof}

\section{Strong convergence analysis}\label{sec:convergence}
This section proves the strong convergence of the coupled tamed Euler--Peano scheme. To this end, we first estimate the defects caused by freezing the coefficients at the left endpoints. Recall the local state increment $\Delta_{h}X(t) = X_{h}(t) - \overline{X}_{h}(t), t \in [0,T]$. We define the total drift and diffusion defects by
\begin{gather}\label{eq:drift-defect-decomposition}
      R_{b,h}(t)
      :=
      b\big(X_{h}(t)\big)
      -
      b_{h}\big(\overline{X}_{h}(t)\big)
      =:
      R_{b,h}^{\mathrm{fr}}(t)
      +
      R_{b,h}^{\mathrm{tm}}(t),
      \\\label{eq:diffusion-defect-decomposition}
      R_{\sigma,h}(t)
      :=
      \sigma\big(X_{h}(t)\big)
      -
      \sigma_{h}\big(\overline{X}_{h}(t)\big)
      =:
      R_{\sigma,h}^{\mathrm{fr}}(t)
      +
      R_{\sigma,h}^{\mathrm{tm}}(t)
\end{gather}
with
\begin{gather*}
      %\label{eq:drift-freezing-defect}
      R_{b,h}^{\mathrm{fr}}(t)
      :=
      b\big(X_{h}(t)\big)
      -
      b\big(\overline{X}_{h}(t)\big),
      \quad
      %\label{eq:drift-taming-defect}
      R_{b,h}^{\mathrm{tm}}(t)
      :=
      b\big(\overline{X}_{h}(t)\big)
      -
      b_{h}\big(\overline{X}_{h}(t)\big),
      \\
      %\label{eq:diffusion-freezing-defect}
      R_{\sigma,h}^{\mathrm{fr}}(t)
      :=
      \sigma\big(X_{h}(t)\big)
      -
      \sigma\big(\overline{X}_{h}(t)\big),
      \quad
      %\label{eq:diffusion-taming-defect}
      R_{\sigma,h}^{\mathrm{tm}}(t)
      :=
      \sigma\big(\overline{X}_{h}(t)\big)
      -
      \sigma_{h}\big(\overline{X}_{h}(t)\big).
\end{gather*}

\begin{lemma}\label{prop:total-consistency-estimates}
      Suppose that Assumptions~\ref{ass:domain} and \ref{ass:coefficients} hold. Let $p\geq1$, and assume that there
      exists an exponent $P$ satisfying
      \begin{equation}\label{eq:total-consistency-moment-condition}
            p\Lambda\big(q_{b},q_{\sigma}\big) < P < p_{\ast}.
      \end{equation}
      Then there exists a constant $C_{p,P,T}>0$, independent of
      $h\in(0,1]$, such that
      \begin{equation}\label{eq:total-drift-consistency}
            \E\left[\int_{0}^{T} |R_{b,h}(t)|^{2p} \,dt\right]
            +
            \E\left[\int_{0}^{T}
            \|R_{\sigma,h}(t)\|_{\mathrm{HS}}^{2p}\,dt\right]
            \leq
            C_{p,P,T}h^{p}.
      \end{equation}
\end{lemma}

\begin{proof}
      We first estimate the freezing defects. By \eqref{eq:drift-polynomial-lipschitz}, we have
      \begin{align*}
            \left|R_{b,h}^{\mathrm{fr}}(t)\right|
            \leq
            L\big(1 + |X_{h}(t)|^{q_{b}} + |\overline{X}_{h}(t)|^{q_{b}}\big)
            |\Delta_{h}X(t)|.
      \end{align*}
      Suppose first that $q_{b}>0$, and define the conjugate exponents $a_{b} := \frac{q_{b}+\gamma}{q_{b}}, a_{b}' := \frac{q_{b}+\gamma}{\gamma}$ satisfying $\frac{1}{a_{b}} + \frac{1}{a_{b}'} = 1$. Applying H\"older's inequality gives
      \begin{align}\label{eq:drift-freezing-holder}
            \E\left[|R_{b,h}^{\mathrm{fr}}(t)|^{2p}\right]
            \leq
            C_{p}\left(\E\big[1 + |X_{h}(t)|^{2p(q_{b}+\gamma)}
            +
            |\overline{X}_{h}(t)|^{2p(q_{b}+\gamma)}\big]\right)^{1/a_{b}}
            \left(\E|\Delta_{h}X(t)|^{2p(q_{b}+\gamma)/\gamma}\right)^{1/a_{b}'}.
      \end{align}
      Since $p(q_{b}+\gamma) < P$, Proposition~\ref{prop:continuous-time-uniform-moment} implies that the first factor on the right-hand side of \eqref{eq:drift-freezing-holder} is uniformly bounded. Noting that
      \begin{equation*}
            \frac{\gamma}{2}\frac{2p(q_{b}+\gamma)}{\gamma}
            =
            p(q_{b}+\gamma)
            <
            P,
      \end{equation*}
      Proposition~\ref{prop:sharp-one-step-increment} yields $\E\big[|\Delta_{h}X(t)|^{2p(q_{b}+\gamma)/\gamma}\big] \leq C_{p,P,T}h^{p(q_{b}+\gamma)/\gamma}$ and consequently
      \begin{align*}
            \left(\E\left[|\Delta_{h}X(t)|^{\frac{2p(q_{b}+\gamma)}{\gamma}}\right]\right)^{1/a_{b}'}
            \leq
            C_{p,P,T}h^{\frac{2p(q_{b}+\gamma)}{\gamma}/(2a_{b}')}
            =
            C_{p,P,T}h^{p}.
      \end{align*}
      Substituting this estimate into \eqref{eq:drift-freezing-holder} gives
      \begin{equation}\label{eq:drift-freezing-pointwise-moment}
            \E\left[|R_{b,h}^{\mathrm{fr}}(t)|^{2p}\right]
            \leq
            C_{p,P,T}h^{p}.
      \end{equation}
      If $q_{b}=0$, \eqref{eq:drift-polynomial-lipschitz} gives $\left|R_{b,h}^{\mathrm{fr}}(t)\right| \leq C|\Delta_{h}X(t)|$. Since $p\gamma < P$, Proposition~\ref{prop:sharp-one-step-increment} with $r=2p$ again gives \eqref{eq:drift-freezing-pointwise-moment}. The diffusion freezing defect is treated in the same way. Indeed, by \eqref{eq:diffusion-polynomial-lipschitz}, one gets
      \begin{align*}
            \left\|R_{\sigma,h}^{\mathrm{fr}}(t)\right\|_{\mathrm{HS}}
            \leq
            L\big(1 + |X_{h}(t)|^{q_{\sigma}}
            +
            |\overline{X}_{h}(t)|^{q_{\sigma}}\big)|\Delta_{h}X(t)|.
      \end{align*}
      Since \eqref{eq:total-consistency-moment-condition} also implies $p(q_{\sigma}+\gamma) < P$, repeating the proceeding arguments gives 
      \begin{equation}\label{eq:diffusion-freezing-pointwise-moment}
            \E\left[\| R_{\sigma,h}^{\mathrm{fr}}(t) \|_{\mathrm{HS}}^{2p} \right] 
            \leq 
            C_{p,P,T}h^{p}. 
      \end{equation}

      We next estimate the taming defects. Owing to \eqref{eq:drift-taming-defect-growth} and \eqref{eq:diffusion-taming-defect-growth}, one gets
      \begin{gather*}
            \left|R_{b,h}^{\mathrm{tm}}(t)\right|
            =
            \left|b\big(\overline{X}_{h}(t)\big)
            - b_{h}\big(\overline{X}_{h}(t)\big)\right|
            \leq
            Ch^{1/2}\left(1 +
            |\overline{X}_{h}(t)|^{\widehat{\gamma}+\gamma_{b}}\right),
            \\
            \left\|R_{\sigma,h}^{\mathrm{tm}}(t)\right\|_{\mathrm{HS}}
            =
            \left\|\sigma\big(\overline{X}_{h}(t)\big)
            - \sigma_{h}\big(\overline{X}_{h}(t)\big)\right\|_{\mathrm{HS}}
            \leq
            Ch^{1/2}\left( 1 +
            |\overline{X}_{h}(t)|^{\widehat{\gamma}+\gamma_{\sigma}}\right).
      \end{gather*}
      Since $\widehat{\gamma}+\gamma_{b} \leq \widehat{\gamma}+\gamma = \Lambda(q_{b},q_{\sigma})$ and $\widehat{\gamma}+\gamma_{\sigma} \leq \widehat{\gamma} + \gamma = \Lambda(q_{b},q_{\sigma})$, 
      \eqref{eq:total-consistency-moment-condition} and Proposition~\ref{prop:continuous-time-uniform-moment} imply that 
      \begin{align}\label{eq:drift-taming-defect-before-moment}
            \E\big[|R_{b,h}^{\mathrm{tm}}(t)|^{2p}\big]
            +
            \E\big[\|R_{\sigma,h}^{\mathrm{tm}}
            (t)\|_{\mathrm{HS}}^{2p}\big]
            \leq
            C_{p,P,T}h^{p}.
      \end{align}

      Finally, by means of the decompositions \eqref{eq:drift-defect-decomposition} and \eqref{eq:diffusion-defect-decomposition}, one has
      \begin{gather*}
            |R_{b,h}(t)|^{2p}
            \leq
            C_{p}\big(\left|R_{b,h}^{\mathrm{fr}}(t)\right|^{2p}
            +
            \left|R_{b,h}^{\mathrm{tm}}(t)\right|^{2p}\big),
            \\
            \|R_{\sigma,h}(t)\|_{\mathrm{HS}}^{2p}
            \leq
            C_{p}\big(\left\|R_{\sigma,h}^{\mathrm{fr}}(t)\right\|_{\mathrm{HS}}^{2p}
            +
            \left\|R_{\sigma,h}^{\mathrm{tm}}(t)\right\|_{\mathrm{HS}}^{2p}\big).
      \end{gather*}
      Together with \eqref{eq:drift-freezing-pointwise-moment}, \eqref{eq:diffusion-freezing-pointwise-moment} and \eqref{eq:drift-taming-defect-before-moment}, 
      we obtain the desired result.
\end{proof}

\subsection{Strong convergence order of the state process}\label{subsec:strong-convergence-state}

We now prove the strong convergence estimate for the state process. A direct application of the Burkholder--Davis--Gundy inequality to the supremum of the error process would introduce an additional diffusion term that is difficult to absorb by the coupled monotonicity condition. To avoid this difficulty, we first derive a stochastic integral inequality at a slightly higher moment order and then apply a standard form of the stochastic Gronwall inequality; see, e.g., \cite{jin2025large, scheutzow2013stochastic}.

\begin{theorem}\label{thm:strong-convergence-state}
      Suppose that Assumptions~\ref{ass:domain} and \ref{ass:coefficients} hold. Let $(X,K)$ be the unique strong solution of
      \eqref{eq:reflected-sde}--\eqref{eq:exact-reflection-support}, and let $(X_{h},K_{h})$ be the coupled tamed Euler--Peano approximation. Let $p \geq 1$ and assume that there exists an exponent $P$ such that $p\Lambda\big(q_{b},q_{\sigma}\big) < P < p_{\ast}$. Then for every $T > 0$, there exists a constant $C_{p,P,T} > 0$, independent of $h \in (0,1]$, such that
      \begin{equation}\label{eq:strong-convergence-state}
            \E\left[ \sup_{0\leq t\leq T}
            |X(t)-X_{h}(t)|^{2p} \right]
            \leq
            C_{p,P,T}h^{p}.
      \end{equation}
\end{theorem}

\begin{proof}
      Let $e_{h}(t) := X(t)-X_{h}(t), t \in [0,T]$ and define the principal coefficient differences
      \begin{equation*}
            \Delta b_{h}(t)
            :=
            b\big(X(t)\big) - b\big(X_{h}(t)\big),
            \quad
            \Delta\sigma_{h}(t)
            :=
            \sigma\big(X(t)\big) - \sigma\big(X_{h}(t)\big).
      \end{equation*}
      It follows from \eqref{eq:drift-defect-decomposition} and \eqref{eq:diffusion-defect-decomposition} that
      \begin{equation*}
            b\big(X(t)\big)
            -
            b_{h}\big(\overline{X}_{h}(t)\big)
            =
            \Delta b_{h}(t)
            +
            R_{b,h}(t),
            \quad
            \sigma\big(X(t)\big)
            -
            \sigma_{h}\big(\overline{X}_{h}(t)\big)
            =
            \Delta\sigma_{h}(t)
            +
            R_{\sigma,h}(t),
      \end{equation*}
      which shows that the error process satisfies
      \begin{align*}
            e_{h}(t)
            =
            \int_{0}^{t} \big(\Delta b_{h}(s) + R_{b,h}(s)\big) \,ds
            +
            \int_{0}^{t} \big(\Delta\sigma_{h}(s) + R_{\sigma,h}(s)\big)\,dW(s)
            +
            K(t)-K_{h}(t).
      \end{align*}
      Noting that there exists an exponent $P$ satisfying $p\Lambda(q_{b},q_{\sigma}) < P < p_{\ast}$, we may choose an exponent $q$ such that
      \begin{equation}\label{eq:auxiliary-error-index}
            p < q 
            <
            \frac{P}{\Lambda(q_{b},q_{\sigma})}
            <
            p_{\ast}
      \end{equation}
      due to $\Lambda(q_{b},q_{\sigma}) \geq 1$. Since $q < p_{\ast} = \eta+\frac12$, one has $2q-1 < 2\eta$. We may therefore choose a constant $\delta_{q}>0$ sufficiently small
      such that
      \begin{equation}\label{eq:delta-q-choice}
            (2q-1)(1+\delta_{q}) < 2\eta.
      \end{equation}
      We next apply It\^{o}'s formula to $|e_{h}(t)|^{2q}$. A standard localization argument is understood in the calculations below. We obtain
      \begin{align}\label{eq:ito-state-error}
            |e_{h}(t)|^{2q}
            =&~
            2q\int_{0}^{t}|e_{h}(s)|^{2q-2}\big\langle
            e_{h}(s),\Delta b_{h}(s)
            +
            R_{b,h}(s)\big\rangle\,ds \nonumber
            \\&~
            +
            q\int_{0}^{t}|e_{h}(s)|^{2q-2}
            \big\|\Delta\sigma_{h}(s)
            +
            R_{\sigma,h}(s)\big\|_{\mathrm{HS}}^{2}\,ds \nonumber
            \\&~
            +
            2q(q-1)\int_{0}^{t}|e_{h}(s)|^{2q-4}
            \big|\big(\Delta\sigma_{h}(s)
            +
            R_{\sigma,h}(s)\big)^{\top}e_{h}(s)\big|^{2}\,ds \nonumber
            \\&~
            +
            2q\int_{0}^{t}|e_{h}(s)|^{2q-2}
            \big\langle e_{h}(s),dK(s)-dK_{h}(s) \big\rangle 
            +
            M_{q,h}(t),
      \end{align}
      where
      \begin{align*}
            M_{q,h}(t)
            :=
            2q \int_{0}^{t} |e_{h}(s)|^{2q-2}
            \big\langle e_{h}(s),
            \big(\Delta\sigma_{h}(s) + R_{\sigma,h}(s)\big)
            \,dW(s)\big\rangle
      \end{align*}
      is a continuous local martingale.

      We first examine the reflection term. By \eqref{eq:exact-reflection-decomposition} and
      \eqref{eq:numerical-reflection-decomposition}, we have
      \begin{align*}
            \big\langle e_{h}(s),dK(s)-dK_{h}(s) \big\rangle
            =
            \big\langle X(s)-X_{h}(s),\mathbf{n}_{X}(s) \big\rangle \,d|K|(s)
            -
            \big\langle X(s)-X_{h}(s),\mathbf{n}_{h}(s) \big\rangle \,d|K_{h}|(s).
      \end{align*}
      Owing to $X_{h}(s) \in \overline{D}$ and $X(s) \in \overline{D}$, the normal-cone characterization \eqref{eq:normal-cone-characterization} gives
      \begin{gather*}
            \big\langle X(s)-X_{h}(s),\mathbf{n}_{X}(s) \big\rangle
            \leq
            0
            \quad
            d|K|\text{-a.e.},
            \\
            \big\langle X(s)-X_{h}(s),\mathbf{n}_{h}(s) \big\rangle
            \geq
            0
            \quad
            d|K_{h}|\text{-a.e.},
      \end{gather*}
      and hence
      \begin{equation}\label{eq:weighted-reflection-error-negative}
            |e_{h}(s)|^{2q-2}
            \big\langle
            e_{h}(s),
            dK(s)-dK_{h}(s)
            \big\rangle
            \leq
            0.
      \end{equation}
      For the quadratic diffusion terms, we use $|A^{\top}x|^{2} \leq \|A\|_{\mathrm{HS}}^{2}|x|^{2}$ to obtain
      \begin{align}\label{eq:quadratic-diffusion-error-bound}
            &~q|e_{h}(s)|^{2q-2}\big\|\Delta\sigma_{h}(s) 
            + R_{\sigma,h}(s)\big\|_{\mathrm{HS}}^{2} \notag
            \\&~+
            2q(q-1)|e_{h}(s)|^{2q-4}
            \big|\big(\Delta\sigma_{h}(s)
            +
            R_{\sigma,h}(s)\big)^{\top}e_{h}(s)\big|^{2} \notag
            \\\leq&~
            q(2q-1)|e_{h}(s)|^{2q-2}
            \big\|\Delta\sigma_{h}(s)
            +
            R_{\sigma,h}(s)\big\|_{\mathrm{HS}}^{2} \notag
            \\\leq&~
            q(2q-1)|e_{h}(s)|^{2q-2}
            \bigg((1+\delta_{q})\|\Delta\sigma_{h}(s)\|_{\mathrm{HS}}^{2}
            +
            \left( 1+\frac{1}{\delta_{q}} \right)
            \|R_{\sigma,h}(s)\|_{\mathrm{HS}}^{2}\bigg),         
      \end{align}
      where Young's inequality has been used. As \eqref{eq:delta-q-choice} ensures $\frac{(2q-1)(1+\delta_{q})}{2} < \eta$, the coupled monotonicity condition \eqref{eq:coupled-monotonicity} yields
      \begin{align}\label{eq:principal-coupled-error-bound}
            &~2q|e_{h}(s)|^{2q-2}
            \big\langle e_{h}(s), \Delta b_{h}(s)\big\rangle
            +
            q(2q-1)(1+\delta_{q})|e_{h}(s)|^{2q-2}
            \|\Delta\sigma_{h}(s)\|_{\mathrm{HS}}^{2} \notag
            \\=&~
            2q |e_{h}(s)|^{2q-2}\Big(\big\langle e_{h}(s),
            \Delta b_{h}(s) \big\rangle
            +
            \frac{(2q-1)(1+\delta_{q})}{2}
            \|\Delta\sigma_{h}(s)\|_{\mathrm{HS}}^{2}\Big) \notag
            \\\leq&~
            2qL|e_{h}|^{2q}.
      \end{align}
      Besides, utilizing Young's inequality again results in
      \begin{align}\label{eq:drift-defect-error-bound}
            2q|e_{h}(s)|^{2q-2}\big\langle
            e_{h}(s),R_{b,h}(s)\big\rangle
            \leq&~
%            2q|e_{h}(s)|^{2q-1}|R_{b,h}(s)|
%            \leq
            C_{q}|e_{h}(s)|^{2q} + C_{q}|R_{b,h}(s)|^{2q},
            \\\label{eq:diffusion-defect-error-bound}
            |e_{h}(s)|^{2q-2}\|R_{\sigma,h}(s)\|_{\mathrm{HS}}^{2}
            \leq&~
            C_{q}|e_{h}(s)|^{2q}
            +
            C_{q}\|R_{\sigma,h}(s)\|_{\mathrm{HS}}^{2q}.
      \end{align}
      Substituting
      \eqref{eq:weighted-reflection-error-negative},
      \eqref{eq:quadratic-diffusion-error-bound},
      \eqref{eq:principal-coupled-error-bound},
      \eqref{eq:drift-defect-error-bound} and
      \eqref{eq:diffusion-defect-error-bound} into
      \eqref{eq:ito-state-error}, we obtain
      \begin{align}\label{eq:error-stochastic-gronwall-form}
            |e_{h}(t)|^{2q}
            \leq
            C_{q}\int_{0}^{t}|e_{h}(s)|^{2q}\,ds
            +
            C_{q}\int_{0}^{t}\big(|R_{b,h}(s)|^{2q}
            +
            \|R_{\sigma,h}(s)\|_{\mathrm{HS}}^{2q}\big)\,ds
            +
            M_{q,h}(t).
      \end{align}

      Define
      \begin{equation*}
            \mathcal{R}_{q,h}(t)
            :=
            C_{q}\int_{0}^{t}\big(|R_{b,h}(s)|^{2q}
            +
            \|R_{\sigma,h}(s)\|_{\mathrm{HS}}^{2q}\big)\,ds.
      \end{equation*}
      Then $\mathcal{R}_{q,h}$ is nonnegative, continuous, adapted, and nondecreasing. Moreover, \eqref{eq:error-stochastic-gronwall-form} takes the form
      \begin{equation*}
            |e_{h}(t)|^{2q}
            \leq
            \mathcal{R}_{q,h}(t)
            +
            C_{q}\int_{0}^{t}|e_{h}(s)|^{2q}\,ds
            +
            M_{q,h}(t).
      \end{equation*}
      Let $\theta := \frac{p}{q} \in (0,1)$ due to \eqref{eq:auxiliary-error-index}. Applying the stochastic Gronwall inequality (see, e.g., \cite[Theorem 4]{scheutzow2013stochastic}) yields
      \begin{align*}
            \E\left[\sup_{0\leq t\leq T}
            |e_{h}(t)|^{2p}\right]
            =
            \E\left[\sup_{0\leq t\leq T}
            \big(|e_{h}(t)|^{2q}\big)^{p/q}\right]
            \leq
            C_{p,q,T}\E\left[
            \mathcal{R}_{q,h}^{p/q}(T)\right]
            \leq
            C_{p,q,T}\left(\E\left[
            \mathcal{R}_{q,h}(T)\right]\right)^{p/q},
      \end{align*}
      where H\"{o}lder's inequality and $q/p > 1$ have been used in the last inequality. By $q\Lambda(q_{b},q_{\sigma}) < P$ in \eqref{eq:auxiliary-error-index}, Lemma~\ref{prop:total-consistency-estimates} can be applied with $q$ in place of $p$. Therefore,
      \begin{align*}
            \E\left[\sup_{0\leq t\leq T}
            |e_{h}(t)|^{2p}\right]
            \leq
            C_{p,q,T}\left(C_{q}
            \E\int_{0}^{T}\Big[|R_{b,h}(s)|^{2q}
            +
            \|R_{\sigma,h}(s)\|_{\mathrm{HS}}^{2q}\Big]\,ds\right)^{p/q}
            \leq
%            C_{p,q,P,T}\big(h^{q}\big)^{p/q}
%            =
            C_{p,q,P,T}h^{p}.
      \end{align*}
      Since the auxiliary exponent $q$ depends only on $p$, $P$ and the growth indices, it may be absorbed into the constant. Thus,
      we obtain \eqref{eq:strong-convergence-state}. 
\end{proof}

\subsection{Strong convergence order of the boundary regulator}\label{subsec:strong-convergence-reflection}
The convergence estimate for the state process also yields a convergence rate for the boundary regulator. Since the coefficient differences are only polynomially locally Lipschitz, the estimate of the reflection error requires a slightly stronger moment condition. We therefore introduce the additional growth index $\overline{q} := \max\big\{ q_{b},q_{\sigma} \big\}$ and define
\begin{equation}\label{eq:reflection-moment-consumption-index}
      \Lambda_{K}\big(q_{b},q_{\sigma}\big)
      :=
      \Lambda\big(q_{b},q_{\sigma}\big)
      +
      \overline{q}.
\end{equation}
Under this strengthened moment condition, we obtain the following strong convergence estimate for the boundary regulator.

\begin{corollary}\label{cor:strong-convergence-reflection}
      Suppose that Assumptions~\ref{ass:domain} and \ref{ass:coefficients} hold. Let $(X,K)$ be the unique strong
      solution of \eqref{eq:reflected-sde}--\eqref{eq:exact-reflection-support}, and let $(X_{h},K_{h})$ be the coupled tamed Euler--Peano approximation. Let $p \geq 1$ and assume that there exists an exponent $P$ such that
      \begin{equation}\label{eq:reflection-strong-moment-condition}
            p\Lambda_{K}\big(q_{b},q_{\sigma}\big)
            <
            P
            <
            p_{\ast}.
      \end{equation}
      Then for every $T>0$, there exists a constant $C_{p,P,T}>0$, independent of $h\in(0,1]$, such that
      \begin{equation}\label{eq:strong-convergence-reflection}
            \E\left[\sup_{0\leq t\leq T}
            |K(t)-K_{h}(t)|^{2p}\right]
            \leq
            C_{p,P,T}h^{p}.
      \end{equation}
\end{corollary}

\begin{proof}
      Recall the state error $e_{h}(t) = X(t)-X_{h}(t)$, and the principal coefficient differences
      \begin{equation*}
            \Delta b_{h}(t)
            =
            b\big(X(t)\big) - b\big(X_{h}(t)\big),
            \quad
            \Delta\sigma_{h}(t)
            =
            \sigma\big(X(t)\big) - \sigma\big(X_{h}(t)\big).
      \end{equation*}
      Subtracting \eqref{eq:coupled-tamed-euler-peano} from \eqref{eq:reflected-sde}, we obtain
      \begin{align}\label{eq:reflection-error-identity}
            K(t)-K_{h}(t)
            =
            e_{h}(t) - \int_{0}^{t}\big(\Delta b_{h}(s)
            + R_{b,h}(s)\big)\,ds
            -
            \int_{0}^{t}\big(\Delta\sigma_{h}(s)
            + R_{\sigma,h}(s)\big)\,dW(s).
      \end{align}
      The strict condition \eqref{eq:reflection-strong-moment-condition} implies $P-p\overline{q} > p\Lambda(q_{b},q_{\sigma}) > 0$. We may therefore choose an exponent $\widetilde{p}$ such that
      \begin{equation}\label{eq:auxiliary-reflection-error-index}
            \frac{pP}{P-p\overline{q}}
            <
            \widetilde{p}
            <
            \frac{P}{\Lambda(q_{b},q_{\sigma})}.
      \end{equation}
      Since $\frac{pP}{P-p\overline{q}} \geq p$, we have $\widetilde{p} > p$. Moreover, \eqref{eq:auxiliary-reflection-error-index} gives $\widetilde{p}\Lambda(q_{b},q_{\sigma}) < P$ and   
      \begin{equation}\label{eq:polynomial-error-holder-condition}
            \frac{p\overline{q}\widetilde{p}}{\widetilde{p}-p} < P.
      \end{equation}
      Applying Theorem~\ref{thm:strong-convergence-state} with $\widetilde{p}$ in place of $p$, we obtain
      \begin{equation}\label{eq:higher-state-error-for-reflection}
            \E\left[\sup_{0\leq t\leq T}
            |e_{h}(t)|^{2\widetilde{p}}\right]
            \leq
            C_{\widetilde{p},P,T}h^{\widetilde{p}}.
      \end{equation}
      We first estimate the principal drift difference. By \eqref{eq:drift-polynomial-lipschitz}, one has
      \begin{align}\label{eq:principal-drift-polynomial-bound}
            |\Delta b_{h}(t)|
            \leq
            L\big(1 + |X(t)|^{q_{b}} + |X_{h}(t)|^{q_{b}}\big)|e_{h}(t)|.
      \end{align}
      Define the conjugate exponents $a := \frac{\widetilde{p}}{\widetilde{p}-p} > 1, a' := \frac{\widetilde{p}}{p} > 1$ satisfying $\frac{1}{a} + \frac{1}{a'} = 1$. Raising \eqref{eq:principal-drift-polynomial-bound} to the power $2p$ and
      applying H\"older's inequality yields
      \begin{align}\label{eq:principal-drift-holder}
            \E\big[|\Delta b_{h}(t)|^{2p}\big]
            \leq
            C_{p}\left(\E\big[1 + |X(t)|^{2p q_{b}a}
            +
            |X_{h}(t)|^{2p q_{b}a}\big]\right)^{1/a}
            \left(\E\big[|e_{h}(t)|^{2pa'}\big]\right)^{1/a'}.
      \end{align}
      By \eqref{eq:polynomial-error-holder-condition} and $q_{b}\leq\overline{q}$, we get $pq_{b}a = \frac{pq_{b}\widetilde{p}}{\widetilde{p}-p} < P$. Theorem~\ref{thm:well-posedness} and Proposition~\ref{prop:continuous-time-uniform-moment} therefore imply that the first factor on the right-hand side of \eqref{eq:principal-drift-holder} is uniformly bounded. Since $2pa' = 2\widetilde{p}$, estimate \eqref{eq:higher-state-error-for-reflection} gives
      \begin{align*}
            \left(\E\big[|e_{h}(t)|^{2pa'}\big]\right)^{1/a'}
            \leq
            \left(\E\left[\sup_{0\leq s\leq T}
            |e_{h}(s)|^{2\widetilde{p}}\right]\right)^{p/\widetilde{p}}
            \leq
            C_{p,P,T}h^{p}.
      \end{align*}
      It follows that
      \begin{equation}\label{eq:principal-drift-error-rate}
            \sup_{0\leq t\leq T}\E
            \big[|\Delta b_{h}(t)|^{2p}\big]
            \leq
            C_{p,P,T}h^{p}.
      \end{equation}
      The same argument, using \eqref{eq:diffusion-polynomial-lipschitz}, gives
      \begin{equation}\label{eq:principal-diffusion-error-rate}
            \sup_{0\leq t\leq T}\E
            \big[\|\Delta\sigma_{h}(t)\|_{\mathrm{HS}}^{2p}\big]
            \leq
            C_{p,P,T}h^{p}.
      \end{equation}
      Indeed, the required state-moment exponent satisfies $\frac{pq_{\sigma}\widetilde{p}}{\widetilde{p}-p} \leq  \frac{p\overline{q}\widetilde{p}}{\widetilde{p}-p} < P$. Taking the supremum over $t \in [0,T]$ in
      \eqref{eq:reflection-error-identity}, and using the elementary inequality $|x_{1}+\cdots+x_{5}|^{2p} \leq C_{p} \sum_{i=1}^{5} |x_{i}|^{2p}$ for $x_{i} \in \mathbb{R}, i = 1,2,\cdots, 5$, we obtain
      \begin{align}\label{eq:reflection-error-five-terms}
            \E\left[\sup_{0\leq t\leq T}|K(t)-K_{h}(t)|^{2p}\right] \notag
            \leq&~
            C_{p}\E\left[\sup_{0\leq t\leq T}|e_{h}(t)|^{2p}\right]
            +
            C_{p}\E\left[\left(\int_{0}^{T}
            |\Delta b_{h}(s)|\,ds\right)^{2p}\right] \nonumber
            \\&~+
            C_{p}\E\left[\left(\int_{0}^{T}
            |R_{b,h}(s)|\,ds\right)^{2p}\right] \nonumber
            +
            C_{p}\E\left[\sup_{0\leq t\leq T}
            \left|\int_{0}^{t}\Delta\sigma_{h}(s)\,dW(s)
            \right|^{2p}\right] \nonumber
            \\&~+
            C_{p}\E\left[\sup_{0\leq t\leq T}
            \left|\int_{0}^{t}R_{\sigma,h}(s)\,dW(s)\right|^{2p}\right].
      \end{align}
      Theorem~\ref{thm:strong-convergence-state} gives
      \begin{equation}\label{eq:reflection-error-state-term}
            \E\left[\sup_{0\leq t\leq T}
            |e_{h}(t)|^{2p}\right]
            \leq
            C_{p,P,T}h^{p}.
      \end{equation}
      By H\"older's inequality and \eqref{eq:principal-drift-error-rate}, we have
      \begin{align}\label{eq:reflection-error-principal-drift-integral}
            \E\left[\left(\int_{0}^{T}|\Delta b_{h}(s)|\,ds\right)^{2p}\right]
            \leq
            T^{2p-1}\int_{0}^{T}\E\big[|\Delta b_{h}(s)|^{2p}\big]\,ds
            \leq
            C_{p,P,T}h^{p}.
      \end{align}
      Similarly, Lemma~\ref{prop:total-consistency-estimates} gives
      \begin{equation}\label{eq:reflection-error-drift-defect-integral}
            \E\left[\left(\int_{0}^{T}|R_{b,h}(s)|\,ds\right)^{2p}\right]
            \leq
            C_{p,P,T}h^{p}.
      \end{equation}
      By the Burkholder--Davis--Gundy inequality, \eqref{eq:principal-diffusion-error-rate}, and H\"older's inequality, one gets
      \begin{align}\label{eq:reflection-error-principal-diffusion}
            \E\left[\sup_{0\leq t\leq T}\left|
            \int_{0}^{t}\Delta\sigma_{h}(s)\,dW(s)\right|^{2p}\right]
            \leq&~
            C_{p}\E\left[\left(\int_{0}^{T}
            \|\Delta\sigma_{h}(s)\|_{\mathrm{HS}}^{2}\,ds\right)^{p}\right] \notag
            \\\leq{}&~
            C_{p,T}\int_{0}^{T}\E\big[
            \|\Delta\sigma_{h}(s)\|_{\mathrm{HS}}^{2p}\big]\,ds  \nonumber
            \\\leq&~
            C_{p,P,T}h^{p}.
      \end{align}
      The same argument, together with Lemma~\ref{prop:total-consistency-estimates}, yields
      \begin{equation}\label{eq:reflection-error-diffusion-defect}
            \E\left[\sup_{0\leq t\leq T}
            \left|\int_{0}^{t}R_{\sigma,h}(s)\,dW(s)\right|^{2p}\right]
            \leq
            C_{p,P,T}h^{p}.
      \end{equation}

      Substituting
      \eqref{eq:reflection-error-state-term},
      \eqref{eq:reflection-error-principal-drift-integral},
      \eqref{eq:reflection-error-drift-defect-integral},
      \eqref{eq:reflection-error-principal-diffusion} and
      \eqref{eq:reflection-error-diffusion-defect} into
      \eqref{eq:reflection-error-five-terms} proves
      \eqref{eq:strong-convergence-reflection}. 
\end{proof}

\section{Numerical experiments}\label{set:experiments}

In this section, we present numerical experiments to illustrate the previous theoretical results. The test equation is a two-dimensional reflected stochastic Ginzburg--Landau-type system on the positive orthant. This example is motivated by reflected stochastic heat equations and SPDEs with hard-wall constraints \cite{nualart1992white, donati1993white}, by hard-wall
Ginzburg--Landau interface models \cite{funaki2001fluctuations, deuschel2007dynamic}, and by reflected Langevin-type dynamics arising in constrained problems \cite{sato2025convergence}. The polynomial state-dependent noise is also in the spirit of stochastic Allen--Cahn and Ginzburg--Landau equations with multiplicative noise \cite{huang2023stability}. We emphasize that the following finite-dimensional system is used as a representative test problem for the coupled monotonicity regime considered in this paper.

Let $D := (0,\infty)^{2}, \overline{D} = [0,\infty)^{2}$ and consider the following RSDE
\begin{equation}\label{eq:numerical-rsde}
      dX(t)
      =
      b\big(X(t)\big)\,dt
      +
      \sigma\big(X(t)\big)\,dW(t)
      +
      dK(t),
      \quad
      t\in[0,T],
\end{equation}
where $X(t) \in \overline{D}$, $W(t)=(W_{1}(t),W_{2}(t))^{\top}$ is a two-dimensional Brownian motion, and $K(t) = (K_{1}(t), K_{2}(t))^{\top}$ is the boundary regulator which keeps the solution in $\overline{D}$. Besides, the coefficients are given by
\begin{equation}\label{eq:numerical-coefficients}
      b(x)
      :=
      \kappa L_{2}x
      +
      \alpha x
      -
      \beta x^{\langle3\rangle},
      \quad
      \sigma(x)
      :=
      \operatorname{diag}
      \big(
      \sigma_{0}+\rho x_{1}^{2},
      \sigma_{0}+\rho x_{2}^{2}
      \big),
\end{equation}
where $\kappa, \alpha, \beta, \sigma_{0} > 0$, $\rho \in \mathbb{R}$ and  
\begin{equation*}
      L_{2}
      :=
      \begin{pmatrix}
            -1 & 1\\
            1 & -1
      \end{pmatrix},
      \quad
      x^{\langle3\rangle}
      :=
      (x_{1}^{3},x_{2}^{3})^{\top},
      \quad
      x=(x_{1},x_{2})^{\top}\in\mathbb{R}^{2}.
\end{equation*}
Here, $\kappa L_{2}X$ represents a nearest-neighbour interaction, while the cubic term $-\beta X^{\langle3\rangle}$ comes from a quartic Ginzburg--Landau-type potential and provides a superlinear dissipative drift. The diffusion coefficient is non-globally Lipschitz due to the quadratic term $\rho x_{i}^{2}$, and the positive constant $\sigma_{0}$ prevents the diffusion from degenerating at the boundary. This example therefore tests the main features covered by our theory: an unbounded convex domain, superlinearly growing coefficients, state-dependent multiplicative noise, and convergence of both the state process and the boundary regulator.

The assumptions of the analysis are satisfied for this example. Indeed, $D = (0,\infty)^{2}$ is a nonempty open convex domain with nonempty boundary. Since $b$ and $\sigma$ are polynomial functions, they are continuous and satisfy the polynomial local Lipschitz estimates
\begin{equation*}
      |b(x)-b(y)|
      \leq
      C\big(1+|x|^{2}+|y|^{2}\big)|x-y|,
      \quad
      \|\sigma(x)-\sigma(y)\|_{\mathrm{HS}}
      \leq
      C\big(1+|x|+|y|\big)|x-y|.
\end{equation*}
Thus $q_b=2$ and $q_\sigma=1$. Moreover, using the negative semidefiniteness of $L_2$ and the inequality $(a+b)^2 \leq \frac{4}{3}(a^2+ab+b^2)$, one obtains
\begin{equation*}
      \big\langle x-y,b(x)-b(y)\big\rangle
      +
      \eta\|\sigma(x)-\sigma(y)\|_{\mathrm{HS}}^{2}
      \leq
      \alpha |x-y|^{2}
\end{equation*}
provided that $\beta\geq \frac{4}{3}\eta\rho^{2}$. In the numerical experiments below, we choose
\begin{equation}\label{eq:2d-GL-parameters}
      \kappa = 0.5,
      \quad
      \alpha = 1,
      \quad
      \rho = 0.5,
      \quad
      \sigma_{0} = 0.5,
      \quad
      \beta = 5.
\end{equation}
Taking $\eta=12$ gives $\frac{4}{3}\eta\rho^{2}=4<5=\beta$ and $p_\ast=\eta+\frac{1}{2}=12.5$. Since $\Lambda(q_b,q_\sigma) = 7$ and
$\Lambda_K(q_b,q_\sigma) = 9$, the choice $P = 10$ satisfies
\begin{equation*}
      7=\Lambda(q_b,q_\sigma)
      <
      9=\Lambda_K(q_b,q_\sigma)
      <
      P=10
      <
      p_\ast=12.5.
\end{equation*}
Therefore the assumptions required for the strong mean-square convergence of both $X_h$ and $K_h$ are fulfilled.

We now apply the coupled tamed Euler--Peano method \eqref{eq:coupled-tamed-euler-peano}--\eqref{eq:numerical-reflection-support}
to \eqref{eq:numerical-rsde} with $T = 1$ and $X_{0} = (0.1,0.2)^{\top}$. Let $t_k=kh$ and set $X_{h,k} := X_h(t_k)$. On each interval $[t_k,t_{k+1}]$, the tamed drift and diffusion coefficients are frozen at $X_{h,k}$, namely 
\begin{equation*}
      a_{h,k} := b_h(X_{h,k}),
      \quad
      \Sigma_{h,k} := \sigma_h(X_{h,k}).
\end{equation*}
Since the reflecting domain is the positive orthant and $\Sigma_{h,k}$ is diagonal, the Skorokhod problem on each time interval
decouples into two one-dimensional reflection problems on $[0,\infty)$. More precisely, for each component $i=1,2$, define the frozen unreflected driving path by
\begin{equation*}
      Y_i(\tau)
      :=
      X_{h,k}^{i}
      +
      a_{h,k}^{i}(\tau-t_k)
      +
      \Sigma_{h,k}^{ii}
      \big(
      W_i(\tau)-W_i(t_k)
      \big),
      \quad
      \tau\in[t_k,t_{k+1}].
\end{equation*}
The one-dimensional Skorokhod correction over $[t_k,t_{k+1}]$ is then
given by
%\begin{equation*}
%      \Delta K_{h,k}^{i}
%      :=
%      \left(
%      -
%      \inf_{\tau\in[t_k,t_{k+1}]}
%      Y_i(\tau)
%      \right)^{+};
%\end{equation*}
\begin{equation*}
      \Delta K_{h,k}^{i}
      :=
      \max\left\{-\inf_{\tau \in [t_k,t_{k+1}]}Y_i(\tau), 0 \right\},
      \quad i = 1,2.
\end{equation*}
see, e.g., \cite{pilipenko2014introduction}. Thus the endpoint update reads
\begin{equation*}
      X_{h,k+1}^{i}
      = 
      Y_i(t_{k+1}) + \Delta K_{h,k}^{i},
      \quad
      K_{h,k+1}^{i}
      =
      K_{h,k}^{i} + \Delta K_{h,k}^{i},
      \quad i = 1, 2.
\end{equation*}
This componentwise implementation preserves the constraint $X_{h,k} \in \overline{D}$ at all grid points. 
% In the numerical computation, the running infimum above is evaluated on the finest temporal grid used to generate the reference Brownian paths. All approximations with different stepsizes are driven by the same Brownian paths, which gives a consistent Monte Carlo estimate of the strong errors.
Figure~\ref{fig:sample-path} displays one representative sample path of the coupled tamed Euler--Peano approximation $X_h(t)$ and the corresponding numerical boundary regulator $K_h(t)$ with $h = 2^{-12}$. The two components of $X_h(t)$ remain nonnegative throughout the simulation, confirming the constraint-preserving property of the method. Moreover, the components of $K_h(t)$ are nondecreasing and increase only when the corresponding component of $X_h(t)$ approaches the boundary, illustrating the action of the reflection term.

\begin{figure}[!htbp]
\begin{center}
      \subfigure[sample paths of $X_{h}(t) = (X_{h}^{1}(t), X_{h}^{2}(t))$]
      {\includegraphics[width=0.45\textwidth]{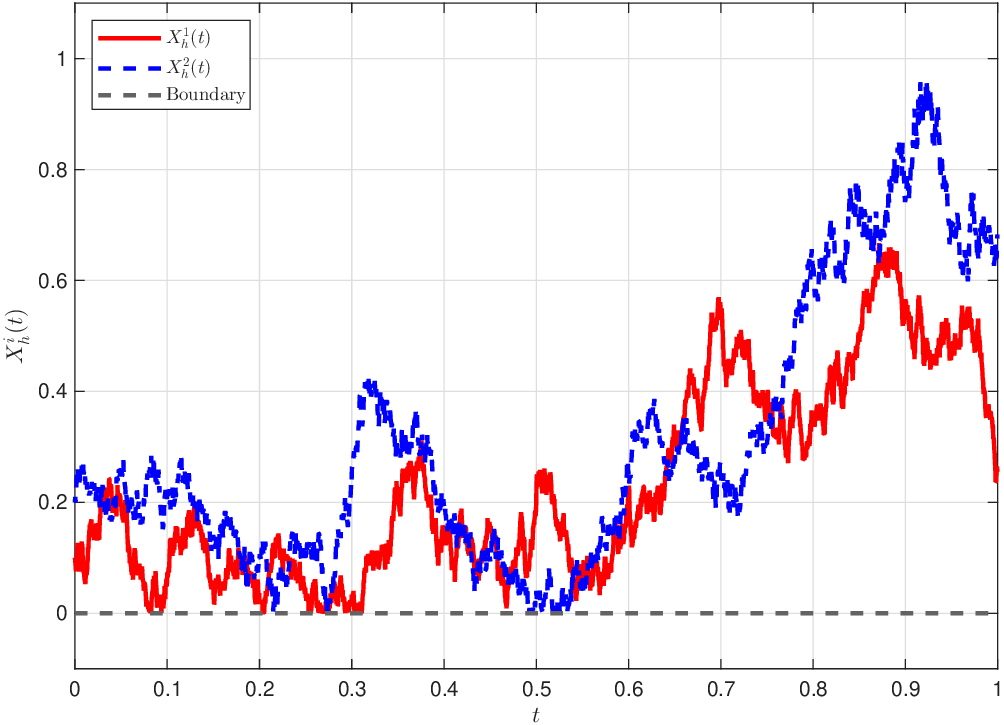}}
      \qquad
      \subfigure[sample paths of $K_{h}(t) = (K_{h}^{1}(t), K_{h}^{2}(t))$]
      {\includegraphics[width=0.45\textwidth]{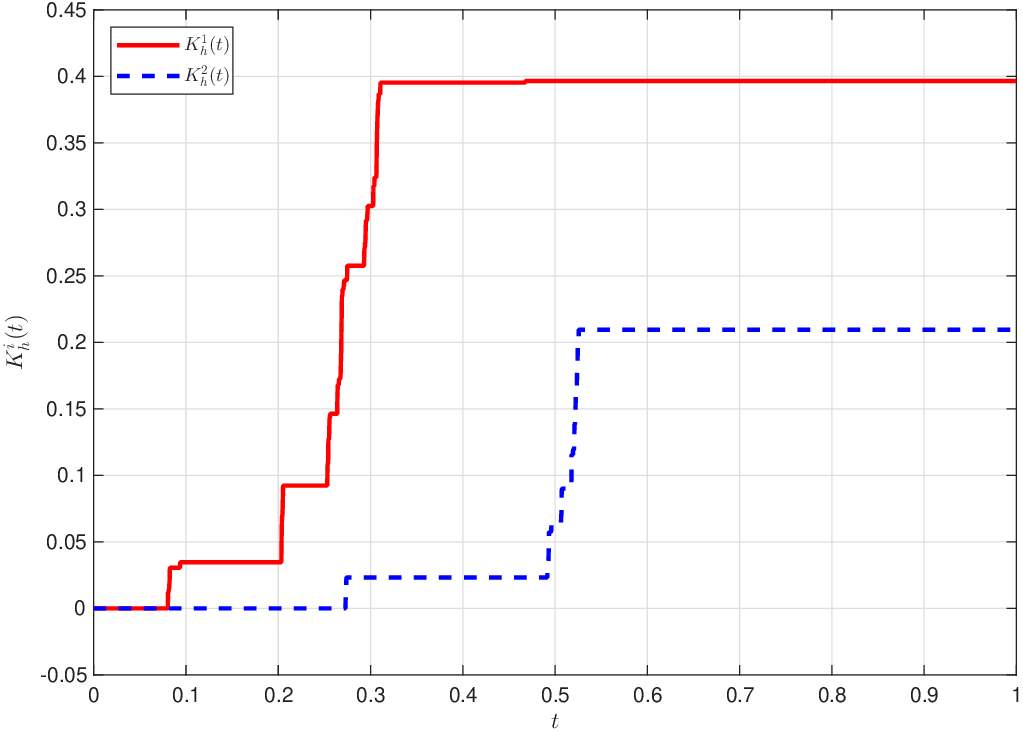}}
      \caption{Sample paths of $X_h(t)$ and $K_h(t)$ for \eqref{eq:numerical-rsde}}
      \label{fig:sample-path}
\end{center}
\end{figure}

To illustrate the strong convergence behaviour of the coupled tamed Euler--Peano method for \eqref{eq:numerical-rsde}, all expectation-type quantities are approximated by the Monte Carlo method with $M = 2000$ independent Brownian sample paths. Since the exact solution is unavailable, we use the reference approximation $X_{h_{\mathrm{ref}}}^{(m)}, K_{h_{\mathrm{ref}}}^{(m)}$ with a finer stepsize $h_{\mathrm{ref}}$ along the $m$-th Brownian sample path. The other numerical approximations $X_{h_{\ell}}^{(m)}, K_{h_{\ell}}^{(m)}$ are calculated by the coupled tamed Euler--Peano method applied to \eqref{eq:numerical-rsde} with six different stepsizes $h_{\ell} = 2^{-\ell},\ell=7,8,\cdots,12$. Then the strong error of the state process and the boundary regulator are measured by
\begin{equation*}
      \mathcal{E}_{X}(h_{\ell})
      :=
      \left(\frac{1}{M}\sum_{m=1}^{M}
      \max_{0\leq k\leq N_{\ell}} \left|X_{h_{\ell}}^{(m)}(t_{k})
      - X_{h_{\mathrm{ref}}}^{(m)}(t_{k})\right|^{2}\right)^{1/2},
\end{equation*}
and
\begin{equation*}
      \mathcal{E}_{K}(h_{\ell})
      :=
      \left(\frac{1}{M}\sum_{m=1}^{M}\max_{0\leq k\leq N_{\ell}}
      \left|K_{h_{\ell}}^{(m)}(t_{k})
      - K_{h_{\mathrm{ref}}}^{(m)}(t_{k})\right|^{2}\right)^{1/2}, 
\end{equation*}
respectively. Figure~\ref{fig:strong-convergence} presents the log--log error plots for the state process and the boundary regulator of the coupled tamed Euler--Peano method applied to  \eqref{eq:numerical-rsde}. The left panel shows the strong error $\mathcal{E}_{X}(h)$, while the right panel shows the strong error $\mathcal{E}_{K}(h)$. In both panels, the error curves decrease as the stepsize $h$ becomes smaller and are close to the reference line of order $1/2$. This indicates that both the state process and the boundary regulator achieve the mean-square strong convergence order $1/2$, in agreement with the theoretical results.

\begin{figure}[!htbp]
\begin{center}
      \subfigure[Order of state process]
      {\includegraphics[width=0.45\textwidth]{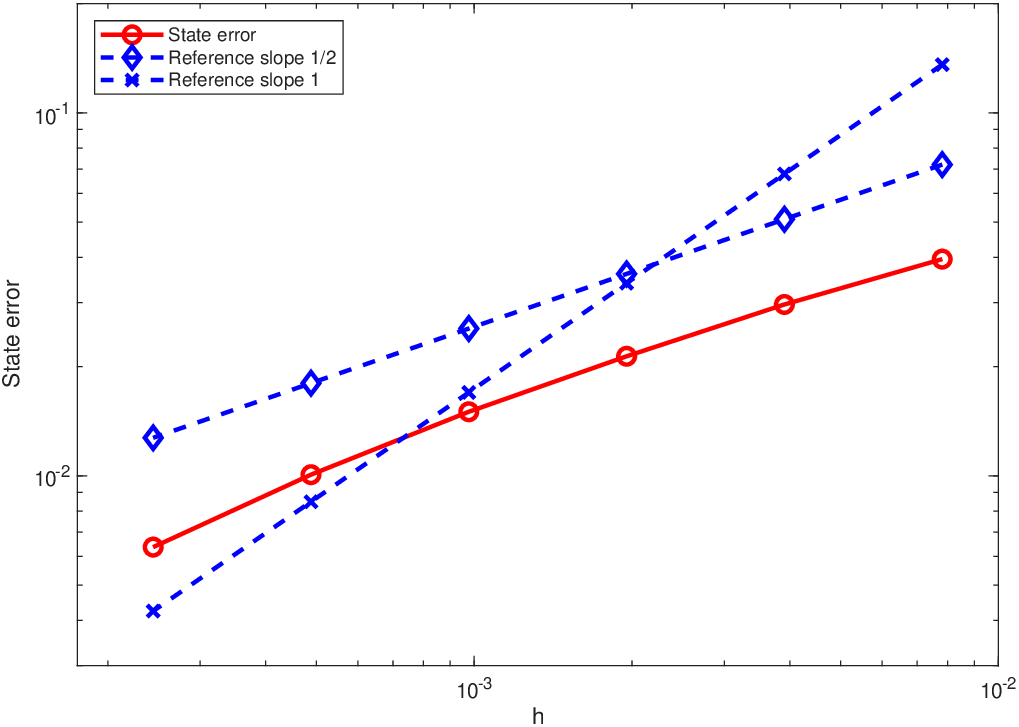}}
      \qquad
      \subfigure[Order of boundary regulator]
      {\includegraphics[width=0.45\textwidth]{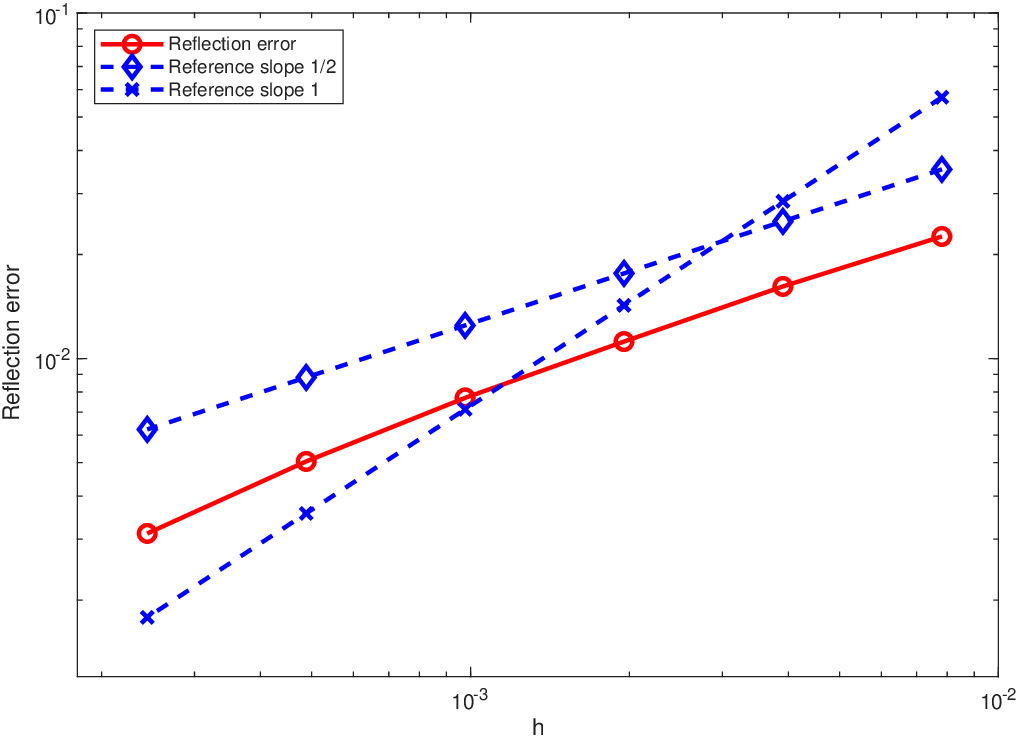}}
      \caption{Strong convergence rates of the coupled tamed Euler--Peano method for  \eqref{eq:numerical-rsde}}
      \label{fig:strong-convergence}
\end{center}
\end{figure}

\section{Conclusion and future work}\label{sec:conclusion}

We proposed a coupled tamed Euler--Peano method for the RSDEs with super-linearly growing drift and diffusion coefficients in possibly unbounded convex domains. Under a coupled monotonicity condition and polynomial local Lipschitz assumptions, we proved the well-posedness of the RSDE, uniform moment estimates for the numerical solution, and strong convergence of order $1/2$ for both the constrained state process and the boundary regulator. Thus, the proposed explicit reflected scheme recovers the standard Euler-type strong order $1/2$ in the reflected super-linear setting while also providing a quantitative order-$1/2$ approximation of the Skorokhod reflection term. The numerical experiments illustrated the constraint-preserving property of the method and supported the theoretical convergence rate.

The present analysis is restricted to normal reflection on convex domains. Several extensions remain open. One natural direction is to investigate whether the coupled taming mechanism can be combined with the geometric conditions used for reflected stochastic differential equations on more general domains, such as the exterior-sphere and non-tangentiality conditions in the sense of
\cite{lions1984stochastic, saisho1987stochastic}. It is also of interest to study oblique reflection, reflected systems driven by jump noise, and higher-order or weak approximation methods. Another challenging problem is to establish quantitative Wong--Zakai
approximation results when the drift and diffusion coefficients are both allowed to grow super-linearly; see, e.g.,
\cite{aida2013wongzakai, aida2014wongzakai}. In that setting, the interaction between the pathwise approximation error, the coupled
dissipativity, and the variation of the boundary regulator requires additional ideas beyond the present Euler--Peano analysis.

%\section{Ackonwlegdes}

%\bibliographystyle{plain}
\bibliographystyle{abbrv}
%\bibliography{CZHBibTeXRSDEs}

\end{document}